\documentclass[english]{article}
\usepackage[T1]{fontenc}
\usepackage[latin1]{inputenc}
\usepackage{geometry}
\usepackage{float}
\usepackage{mathrsfs}
\usepackage{amsmath}
\usepackage{amssymb}
\usepackage{graphicx}
\usepackage{subfigure}

\makeatletter

\floatstyle{ruled}
\newfloat{algorithm}{tbp}{loa}
\providecommand{\algorithmname}{Algorithm}
\floatname{algorithm}{\protect\algorithmname}

\usepackage{amsthm}

\usepackage{mathrsfs}

\usepackage{amsfonts}

\usepackage{epsfig}

\usepackage{bm}

\usepackage{mathrsfs}

\usepackage{enumerate}

\@ifundefined{definecolor}{\@ifundefined{definecolor}
 {\@ifundefined{definecolor}
 {\usepackage{color}}{}
}{}
}{}

\usepackage{subfig}\usepackage[all]{xy}

\newtheorem{theorem}{Theorem}[section]
\newtheorem{lem}{Lemma}[section]
\newtheorem{rem}{Remark}[section]
\newtheorem{prop}{Proposition}[section]

\newcounter{hypA}
\newenvironment{hypA}{\refstepcounter{hypA}\begin{itemize}
  \item[({\bf A\arabic{hypA}})]}{\end{itemize}}
\newcounter{hypB}

\newcounter{hypD}

\usepackage{babel}\date{}

\usepackage{babel}

\makeatother

\usepackage{babel}

\newcommand{\bbE}{\mathbb{E}}

\begin{document}

\begin{center}

{\Large \textbf{Unbiased Filtering of a Class of Partially Observed Diffusions}}

\vspace{0.5cm}

BY AJAY JASRA$^{1}$, KODY J.~H. LAW$^{2}$ \& FANGYUAN YU$^{1}$

{\footnotesize $^{1}$Computer, Electrical and Mathematical Sciences and Engineering Division, King Abdullah University of Science and Technology, Thuwal, 23955, KSA.}
{\footnotesize E-Mail:\,} \texttt{\emph{\footnotesize ajay.jasra@kaust.edu.sa, fangyuan.yu@kaust.edu.sa}}\\
{\footnotesize $^{2}$School of Mathematics,
University of Manchester, Manchester, M13 9PL, UK.}
{\footnotesize E-Mail:\,} \texttt{\emph{\footnotesize kodylaw@gmail.com}}

\end{center}

\begin{abstract}
In this article we consider a Monte Carlo-based method to filter partially observed diffusions observed at regular and discrete times.
Given access only to Euler discretizations of the diffusion process, 
we present a new procedure which can return online estimates of the 
filtering distribution 
with no time discretization bias and 
finite variance. Our approach is based upon a novel double application
of the randomization methods of \cite{rhee} along with the multilevel particle filter (MLPF) approach of \cite{mlpf}. 
A numerical comparison of our
new approach with the MLPF, on a single processor,
shows that similar errors are possible for a mild increase in computational cost. 
However, the new method scales strongly to arbitrarily many processors.  \\
\noindent \textbf{Key words}: Partially Observed Diffusions; Randomization Methods; Multilevel Monte Carlo.
\end{abstract}

\section{Introduction}

We consider the problem of estimating 
a hidden diffusion process, 
given access to only discrete 
time observations. 
It is assumed that the observation process is independent of all other random variables
when conditioned on the hidden process at a given time.
Such a model is often termed a hidden Markov or state-space model in the literature (e.g.~\cite{cappe}) and has many real applications in engineering,
finance and economics. 

We are particularly concerned in the filtering problem:
estimating the diffusion process online, that is, recursively in time as data arrive.
Particle filters (PFs)  
are numerical methods which can provide exact approximations 
(consistent in the Monte Carlo sample size) of filtering distributions associated to 
state space models,
with a fixed computational cost per observation time
(see e.g.~\cite{delm:13} and
the references therein). 
The method can sometimes provide errors that are uniform in time (e.g.~\cite{delm:13})
and is most effective when the hidden state is in moderate
dimension (1-15).
In the context of diffusions, the problem is even more challenging than usual, 
because the transition density of the diffusion process is seldom available up-to a non-negative
and unbiased estimator, which precludes 
the use of exact simulation methods such as in \cite{fearn}. 
As a result, it is common to 
adopt 
a time-discretization of the diffusion process, 
for instance using the Euler method, 
and perform inference using a PF with this biased model. 
This 
approach can be further enhanced by using a PF version of the popular
multilevel Monte Carlo (MLMC) method of \cite{giles,heinrich}, 
called the multilevel particle filter (MLPF); see e.g.~\cite{mlpf_new,cpf_clt,mlpf}. 
The basic notion of this methodology is to introduce a hierarchy of time-discretized 
filters and 
a collapsing sum representation of the most precise time-discretization, 
and then to approximate the representation using independent coupled
particle filters (CPFs). 
Using this approach, the cost to achieve a given mean square
error (MSE) can be reduced quite significantly relative to a single level strategy,
under appropriate assumptions. 
However, we note that this method will still produce estimates with 
a bias from the most precise time-discretization.

The objective of the present article is to develop a technique that can remove the time-discretization bias from filtering, 
even when we cannot sample from the exact unobserved
diffusion and we do not have access to an estimate of the transition density 
which is non-negative and unbiased (as in \cite{fearn}). 
The approach we follow is to consider randomization
schemes as for instance found in \cite{mcl,rhee,vihola}.
In the context of estimating a class of expectations w.r.t.~laws of diffusion processes, 
\cite{rhee} show how to obtain
unbiased estimates with only access to 
time-discretized approximations of the diffusion process. 
In terms of cost to obtain
a given MSE (variance), it can perform better or worse than the MLMC method, 
depending on the context; 
the improvement depends upon both the time discretization 
and some underlying properties of the diffusion of interest.
This approach cannot be routinely extended to the filtering of diffusion processes 
because one cannot in general obtain independent and online (fixed computational cost per observation time step) samples even from the discrete-time approximations. 
Here we develop a novel double
randomization scheme, based upon that in \cite{rhee}, 
which uses the MLPF methodology above to yield unbiased and finite variance estimators of the filter. Moreover, our method is
intrinsically parallelizable and has a computational cost that is comparable to the MLPF.  This latter point is investigated both mathematically and empirically with numerical simulations. 
To the best of our knowledge, this is one of the first methods to achieve unbiasedness, when exact simulations of diffusions may not be possible.
Moreover, the assumptions on the diffusion which are actually required to implement the procedure are minimal.

This article is structured as follows. In Section \ref{sec:app} we present our approach along with the details of the problem of interest. Alternative approaches are also discussed.
In Section \ref{sec:theory} we demonstrate that our estimate is both unbiased and of finite variance. 
In Section \ref{sec:numerics} our numerical results are presented. The proofs of our results in Sections \ref{sec:app} and \ref{sec:theory} are given in Appendix \ref{sec:proofs}.

\section{Approach}\label{sec:app}

\subsection{Notations}

Let $(\mathsf{X},\mathcal{X})$ be a measurable space.
For $\varphi:\mathsf{X}\rightarrow\mathbb{R}$ we write $\mathcal{B}_b(\mathsf{X})$, $\mathcal{C}^2(\mathsf{X})$
and $\textrm{Lip}(\mathsf{X})$ as the collection of bounded measurable, continuous, twice-differentiable and Lipschitz (finite Lipschitz constants) functions respectively.
If $\mathsf{X}\subseteq\mathbb{R}^d$ then the metric for $\varphi\in \textrm{Lip}(\mathsf{X})$ is $L_2$ (i.e.~$|\varphi(x)-\varphi(x')|\leq C \| x - x' \|$ and $\| \cdot \|$ is the $L_2-$norm).
For $\varphi\in\mathcal{B}_b(\mathsf{X})$, we write the supremum norm $\|\varphi\|_{\infty}=\sup_{x\in\mathsf{X}}|\varphi(x)|$.
$\mathcal{P}(\mathsf{X})$  denotes the collection of probability measures on $(\mathsf{X},\mathcal{X})$.
For a finite measure $\mu$ on $(\mathsf{X},\mathcal{X})$
and a $\varphi\in\mathcal{B}_b(\mathsf{X})$, the notation $\mu(\varphi)=\int_{\mathsf{X}}\varphi(x)\mu(dx)$ is used.
For $(\mathsf{X}\times\mathsf{Y},\mathcal{X}\vee\mathcal{Y})$ a measurable space and $\mu$ a non-negative finite measure on this space,
we use the tensor-product of function notations for $(\varphi,\psi)\in\mathcal{B}_b(\mathsf{X})\times\mathcal{B}_b(\mathsf{X})$,
$\mu(\varphi\otimes\psi)=\int_{\mathsf{X}\times\mathsf{Y}}\varphi(x)\psi(x')\mu(d(x,x'))$,
where the notation $d(x,x')=du$ is used for $u=(x,x')$.
Let $K:\mathsf{X}\times\mathcal{X}\rightarrow[0,a]$, $0<a<+\infty$ be a non-negative kernel and $\mu$ be a measure then we use the notations
$
\mu K(dx') = \int_{\mathsf{X}}\mu(dx) K(x,dx')
$
and for $\varphi\in\mathcal{B}_b(\mathsf{X})$, 
$
K(\varphi)(x) = \int_{\mathsf{X}} \varphi(x') K(x,dx').
$
For $A\in\mathcal{X}$ the indicator is written $\mathbb{I}_A(x)$. $\mathbb{Z}^+$ are the non-negative integers. 
$\mathcal{N}(\mu,\sigma^2)$ 
denotes a one dimensional Gaussian distribution of mean $\mu$ and variance $\sigma^2$. 
$\mathcal{L}(\mu,\sigma)$ denotes a Laplace distribution with location $\mu$ and scale $\sigma$.

\subsection{General Problem}\label{sec:gen}

We begin by presenting the idea at a very high level, which will help to motivate the methodology to be described in the context of partially observed diffusions.
Let $\eta\in\mathcal{P}(\mathsf{X})$ be a probability measure of interest. We assume that we only are able to work with biased versions of $\eta$, 
for example arising from discretization with time-step 
$\Delta_l$. 
In particular, consider
$\{\eta^l\}_{l\in\mathbb{Z}^+}$, $\eta^l\in\mathcal{P}(\mathsf{X})$ such that for any $\varphi\in\mathcal{B}_b(\mathsf{X})$: 
\begin{equation}\label{eq:weak_error_gen}
\lim_{l\rightarrow\infty}\eta^l(\varphi)=\eta(\varphi).
\end{equation}

Our objective is to introduce a Monte Carlo method that can deliver unbiased and finite variance estimates of $\eta(\varphi)$, which in principle could
be achieved in the following manner, using randomization approaches (e.g.~\cite{rhee}, see also \cite[Theorem 3]{vihola}). 
Suppose that one can produce
a sequence of independent random variables $(\Xi_l)_{l\in\mathbb{Z}^+}$ such that 
\begin{equation}\label{eq:xi_cond}
\mathbb{E}[\Xi_l] = 
\eta^l(\varphi) - \eta^{l-1}(\varphi) \, .
\end{equation}
where  $\eta^{-1}(\varphi):=0$.
Let $\mathbb{P}_L(l)$ be a positive probability mass function on $\mathbb{Z}^+$. 
Sample $L$ from $\mathbb{P}_L$ and consider the estimate 
\begin{equation}\label{eq:eta_singleterm}
\widehat{\eta(\varphi)}_1 = 
\frac{\Xi_L}{\mathbb{P}_L(L)} \, .
\end{equation}
By \cite[Theorem 3]{vihola}, $\widehat{\eta(\varphi)}_1$ is an unbiased and finite variance estimator of $\eta(\varphi)$ if
\begin{equation}\label{eq:ub_fv_cond}
\sum_{l\in\mathbb{Z}^+}\frac{1}{\mathbb{P}_L(l)}\mathbb{E}[\Xi_l^2] < +\infty \, .
\end{equation}
Notice the very important and very powerful fact 
that once we can obtain \eqref{eq:eta_singleterm}, 
we can construct independent and identically distributed (i.i.d.) 
samples in parallel by sampling $L_i$ independently from $\mathbb{P}_L$ for 
$i\in\{1,\dots, M\}$. 
From these samples we are able to construct an unbiased estimator with 
$\mathcal{O}(M^{-1})$ mean squared error 
(assuming condition \eqref{eq:ub_fv_cond} holds) as follows
$$
\widehat{\eta(\varphi)} = \frac1M \sum_{i=1}^M \frac{\Xi_{L_i}}{\mathbb{P}_L(L_{i})} \, .
$$

The main challenge in our application is to deliver an algorithm which can provide unbiased estimates of $\eta^l(\varphi)$, $(\Xi_l)_{l\in\mathbb{Z}^+}$ and $\mathbb{P}_L$ such that \eqref{eq:xi_cond} and \eqref{eq:ub_fv_cond} are satisfied. We will construct such a method in our particular problem of interest.

\subsection{Partially Observed Diffusions}

The following presentation follows \cite{mlpf}.
We start with a diffusion process:
\begin{eqnarray}
dZ_t & = & a(Z_t)dt + b(Z_t)dW_t \, ,
\label{eq:sde}
\end{eqnarray}
with $Z_t\in\mathbb{R}^d=\mathsf{X}$, 
$a:\mathbb{R}^d\rightarrow\mathbb{R}^d$ ($j$th element denoted $a^j$), $b:\mathbb{R}^d\rightarrow\mathbb{R}^{d\times d}$ ($(j,k)$th element denoted $b^{j,k}$), 
$t\geq 0$ and $\{W_t\}_{t\geq 0}$ a $d-$dimensional Brownian motion. 
The following assumptions are made throughout the article. We set $Z_0=x^*\in\mathsf{X}$.

\begin{quote} 
The coefficients $a^j, b^{j,k} \in \mathcal{C}^2(\mathsf{X})$, for $j,k= 1,\ldots, d$. 
Also, $a$ and $b$ satisfy 
\begin{itemize}
\item[(i)] {\bf uniform ellipticity}: $b(z)b(z)^T$ is uniformly positive definite 
over $z\in \mathsf{X}$;
\item[(ii)] {\bf globally Lipschitz}:
there is a $C>0$ such that 
$|a^j(z)-a^j(z')|+|b^{j,k}(z)-b^{j,k}(z')| \leq C |z-z'|$ 
for all $(z,z') \in \mathsf{X}\times\mathsf{X}$, $(j,k)\in\{1,\dots,d\}^2$. 
\end{itemize}
\end{quote}
It is remarked that these assumptions are made for mathematical convenience. In general, all one really requires from a numerical perspective,
is the existence of the solution of the stochastic differential equation. However, one must note that the numerical performance will be affected by the
properties of the diffusion process.

The data are observed at regular unit time-intervals (i.e.~in discrete time) 
$(y_1,y_2,\dots)$, where $y_k \in \mathsf{Y}$.
It is assumed that conditional on $Z_{k}$, 
$Y_k$ is independent of all other random variables with strictly positive density $G(z_{k},y_k)$. Let $M(z,z')$  be the transition density (assumed to exist) of the diffusion process (over unit time)
and consider a discrete-time Markov chain $(X_0,X_1,\dots)$ 
with initial distribution $M(x^*,\cdot)$ and transition density $M(x,\cdot)$. Here we are
creating a discrete-time Markov chain that corresponds to the discrete-time skeleton of the diffusion process at a time lag of 1.
Since we condition on a given realization of observations $(y_1,y_2,\dots)$,
we will henceforth write $G_{k}(x_{k})$ instead of $G(x_{k},y_{k+1})$, 
and it is assumed that $G_{k}\in\mathcal{B}_b(\mathsf{X})$ for every $k\geq 0$.
Then we define, for $B\in\mathcal{X}$
$$
\gamma_{n}(B) := \int_{\mathsf{X}^{n+1}}\mathbb{I}_B(x_n)\Big(\prod_{p=0}^{n-1} G_{p}(x_{p})\Big)M(x^*,x_{0})\prod_{p=1}^n M(x_{p-1},x_{p}) dx_{0:n}.
$$
The predictor is $\eta_{n}(B)=\gamma_{n}(B)/\gamma_{n}(1)$ which corresponds
to the distribution associated to $X_{n} := Z_{n}|y_1,\dots,y_{n-1}$. The filter is
$$
\bar{\eta}_n(B) = \frac{\eta_n(G_n\mathbb{I}_B)}{\eta_n(G_n)}.
$$

Consider an Euler discretization of the diffusion
with discretization $\Delta_l=2^{-l}$, $l\geq 0$ and write the associated transition over unit time as $M^l(x,x')$. Then we define, for $B\in\mathcal{X}$
$$
\gamma_{n}^l(B) := \int_{\mathsf{X}^{n+1}}\mathbb{I}_B(x_n)\Big(\prod_{p=0}^{n-1} G_{p}(x_{p})\Big)M^l(x^*,x_{0})\prod_{p=1}^n M^l(x_{p-1},x_{p}) dx_{0:n}.
$$
The predictor is $\eta_{n}^l(B)=\gamma_{n}^l(B)/\gamma_{n}^l(1)$ which corresponds
to the Euler approximation $X_n^{l}$
of the distribution associated to $Z_{n}|y_1,\dots,y_{n-1}$. The filter is
$$
\bar{\eta}_n^l(B) = \frac{\eta_n^l(G_n\mathbb{I}_B)}{\eta_n^l(G_n)}.
$$

Let $l\in\mathbb{N}$.
Throughout the article, it is assumed that there exists a Markov kernel 
$\check{M}^l:\mathsf{X}\times\mathsf{X}\rightarrow\mathcal{P}(\mathsf{X}\times\mathsf{X})$
such that for any $B\in\mathcal{X}$, $(x,x')\in\mathsf{X}\times\mathsf{X}$:
$$
\check{M}^l(B\times \mathsf{X})(x,x') = M^l(B)(x) \qquad \check{M}^l(\mathsf{X}\times B)(x,x') = M^{l-1}(B)(x').
$$
This can be done by the coupling scheme used in \cite{mlpf} for example and is the one used in this article. 
We will use the notation $\check{\eta}_0^l(d(x,x'))$ 
as an exchangeable notation for $\check{M}^l((x^*,x^*), d(x,x'))$.

We note the following (the proof is in Appendix \ref{app:disc_conv}), 
which verifies that \eqref{eq:weak_error_gen} will hold in our context.
\begin{prop}\label{prop:euler_pred_conv}
For any $n\in\mathbb{Z}^+$, $\varphi\in\mathcal{B}_b(\mathsf{X})$
$$
\lim_{l\rightarrow\infty}\bar{\eta}_n^l(\varphi)=\bar{\eta}_n(\varphi).
$$
\end{prop}

\subsection{Strategy}

We will now detail how one can obtain $(\Xi_l)_{l\in\mathbb{Z}^+}$ (as in Section \ref{sec:gen}) via unbiased estimates of $\bar{\eta}_n^0(\varphi)$ and of $[\bar{\eta}_n^l-\bar{\eta}_n^{l-1}](\varphi)$, $\varphi\in\mathcal{B}_b(\mathsf{X})$ and
$l\in\mathbb{Z}^+$ fixed. This idea is based upon finding
\emph{biased} but consistent (in the Monte Carlo sample size) estimates of $\bar{\eta}_n^0(\varphi)$ and $[\bar{\eta}_n^l-\bar{\eta}_n^{l-1}](\varphi)$.  
Our strategy is as follows.
Let $(N_p)_{p\in\mathbb{Z}^+}$, $N_p\in\mathbb{Z}^+$  be an increasing sequence of positive integers, with $\lim_{p\rightarrow\infty} N_p=\infty$. 
Let $\bar{\eta}_n^{N_p,0}(\varphi)$ (resp.~$[\bar{\eta}_n^{N_p,l}-\bar{\eta}_n^{N_p, l-1}](\varphi)$) be a Monte Carlo (type) estimate of $\bar{\eta}_n^0(\varphi)$ (resp.~$[\bar{\eta}_n^l-\bar{\eta}_n^{l-1}](\varphi)$), of $N_p$ samples.
By the consistency of the approximations, almost surely, we have
$$
\lim_{p\rightarrow\infty}\bar{\eta}_n^{N_p,0}(\varphi) = \bar{\eta}_n^0(\varphi) \, ,
\quad {\rm and} \quad  
\lim_{p\rightarrow\infty}\{[\bar{\eta}_n^{N_p,l}-\bar{\eta}_n^{N_p, l-1}](\varphi)\}=[\bar{\eta}_n^l-\bar{\eta}_n^{l-1}](\varphi)) \, .$$
We do not require that $\mathbb{E}[\bar{\eta}_n^{N_p,0}(\varphi)] = \bar{\eta}_n^0(\varphi)$ or that $\mathbb{E}[[\bar{\eta}_n^{N_p,l}-\bar{\eta}_n^{N_p, l-1}](\varphi)]=\bar{\eta}_n^l(\varphi)-\bar{\eta}_n^{l-1}(\varphi)$ and we use the convention $\bar{\eta}_n^{N_{-1},0}(\varphi):=0$, 
$[\bar{\eta}_n^{N_{-1},l}-\bar{\eta}_n^{N_{-1}, l-1}](\varphi):=0$.
Below we explain how the estimates of $\bar{\eta}_n^{N_p,0}(\varphi)$ and 
$[\bar{\eta}_n^{N_p,l}-\bar{\eta}_n^{N_p, l-1}](\varphi)$ can be used to obtain unbiased 
estimates of $\bar{\eta}_n^0(\varphi)$ and of $[\bar{\eta}_n^l-\bar{\eta}_n^{l-1}](\varphi)$, 
hence the random variables 
$(\Xi_l)_{l\in\mathbb{Z}^+}$ (as in Section \ref{sec:gen}). 
Suppose that one has a positive probability mass function $\mathbb{P}_P(p)$ on $p\in\mathbb{Z}^+$.
Define
$$
\Xi_{l,p} :=  \left\{\begin{array}{ll}
 \frac{1}{\mathbb{P}_P(p)}[\bar{\eta}_n^{N_p,0}-\bar{\eta}_n^{N_{p-1},0}](\varphi) & \textrm{if}~l=0 \\
\frac{1}{\mathbb{P}_P(p)}\Big([\bar{\eta}_n^{N_p,l}-\bar{\eta}_n^{N_p, l-1}](\varphi)  -
[\bar{\eta}_n^{N_{p-1},l}-\bar{\eta}_n^{N_{p-1}, l-1}](\varphi)
\Big)& \textrm{otherwise} \, .
\end{array}\right .
$$
Now set $\Xi_l =  \Xi_{l,P}$, 
where $P$ is sampled according to $\mathbb{P}_P(p)$.
Again using \cite[Theorem 3]{vihola}, we have that
$$
\mathbb{E}[\Xi_l] =  \left\{\begin{array}{ll}
\bar{\eta}_n^{0}(\varphi) & \textrm{if}~l=0 \\
\bar{\eta}_n^{l}(\varphi)-\bar{\eta}_n^{l-1}(\varphi) & \textrm{otherwise} \, .
\end{array}\right.
$$
For each $l\in\mathbb{Z}^+$, $\Xi_l$ has finite variance provided one has
\begin{equation}\label{eq:xipl_cond}
\bbE [\Xi_l^2] = 
\sum_{p\geq 0} \mathbb{P}_P(p) \bbE [\Xi_{l,p}^2 ] <+\infty \, .
\end{equation}
The main objective is now to consider how one can compute 
$\bar{\eta}_n^{N_p,0}(\varphi)$ and $[\bar{\eta}_n^{N_p,l}-\bar{\eta}_n^{N_p, l-1}](\varphi)$. 
We will comment on the various aspects of the approach in Section \ref{sec:algo}.

\subsubsection{Approximating $\bar{\eta}_n^0(\varphi)$}

We will now build a procedure, based upon particle filters, to approximate $\eta_n^0$ which will use $N_0,N_1,\dots$ samples.
First we explain the PF with a fixed number of particles $N$.
The objective of PFs is to recursively in $n$ approximate $\eta_n^0$. Let $n\in\mathbb{N}$, $B\in\mathcal{X}$ and $\mu\in\mathcal{P}(\mathsf{X})$ and define the probability measure:
$$
\Phi_n^0(\mu)(B) = \frac{\mu(G_{n-1}M^0(B))}{\mu(G_{n-1})}.
$$
Note that for any $n\in\mathbb{N}$, $\eta_n^0(B) = \Phi_n^0(\eta_{n-1}^0)(B)$. The PF at time $n$ has law:
$$
\mathbb{P}^N(d(x_0^{1:N,0},\dots,x_n^{1:N,0})) = \Big(\prod_{i=1}^N \eta_0^0(dx_0^{i,0})\Big)\Big(\prod_{k=1}^n\prod_{i=1}^N\Phi_k^0(\eta_{k-1}^{N,0})(dx_k^{i,0})\Big)
$$
where
$$
\eta_{k-1}^{N,0}(dx) = \frac{1}{N}\sum_{i=1}^N \delta_{x_{k-1}^{i,0}}(dx).
$$
The PF is summarized in Algorithm \ref{alg:pf}.

\begin{algorithm}[h!]
\begin{enumerate}
\item{Initialization: For $i\in\{1,\dots,N\}$ sample $X_0^{i,0}$ from $\eta_0^0$. Set $n=1$.}
\item{Resampling and Sampling: For $i\in\{1,\dots,N\}$ sample $X_n^{i,0}$ from $\Phi_n^0(\eta_{n-1}^{N,0})$. 
Set $n=n+1$ and return to the start of 2.}
\end{enumerate}
\caption{A Particle Filter with a fixed number of samples $N\in\mathbb{Z}^+$.}
\label{alg:pf}
\end{algorithm}

To form our approximation of $\bar{\eta}_n^0(\varphi)$ with $N_0$ samples, we run the PF as described above with $N_0$ samples. To form the approximation
with $N_1$ samples, we run a PF independently of the first PF with $N_1-N_0$ samples and so on, for any $p\geq 2$ (i.e.~with $N_2-N_1,\dots, N_{p}-N_{p-1}$ samples).
For a given $p$, the joint law of the simulated samples is
$$
\mathbb{P}_c(d(x_0^{1:N_p,0},\dots,x_n^{1:N_p,0})) = \prod_{q=0}^{p} \mathbb{P}^{N_q-N_{q-1}}(d(x_0^{N_{q-1}+1:N_q,0},\dots,x_n^{N_{q-1}+1:N_q,0}))
$$
where  $N_{-1}:=0$. Expectations w.r.t.~$\mathbb{P}_c$ will be written $\mathbb{E}_c$. 
Define now the notation 
\begin{eqnarray*}
\eta_{n}^{N_{0:p},0}(\varphi) & := & \sum_{q=0}^p \Big(\frac{N_q-N_{q-1}}{N_p}\Big)\eta_{n}^{N_q-N_{q-1},0}(\varphi) \, ,\\
\eta_{n}^{N_q-N_{q-1},0}(\varphi) & := & \frac{1}{N_{q}-N_{q-1}}\sum_{i=N_{q-1}+1}^{N_q}\varphi(x_n^{i,0}) \, .
\end{eqnarray*}
Here $x_n^{1,0},\dots,x_n^{N_{0},0}$ are generated from the first PF, independently $x_n^{N_0+1,0},\dots,x_n^{N_{1},0}$ from the second and so on.
The procedure for sampling, in order to compute \eqref{eq:eta_0_app} is summarized in Algorithm \ref{alg:pf_est}.
The approximation of $\bar{\eta}_n^0(\varphi)$ is finally 
\begin{equation}\label{eq:eta_0_app}
\bar{\eta}_n^{N_p,0}(\varphi) = 
\frac{\eta_{n}^{N_{0:p}0}(G_n\varphi)}{\eta_{n}^{N_{0:p},0}(G_n)} \, . 
\end{equation}
We remark that the strategy of running a CPF with $N_{p-1}$ samples and then running an additional one with $N_p - N_{p-1}$
would not suffice and lead to biased estimator, hence the strategy adopted,

\begin{algorithm}
\begin{enumerate}
\item{Initialization: Run Algorithm \ref{alg:pf} with $N_0$ samples. Set $q=1$. 
If $p=0$ stop; otherwise go to 2.}
\item{Iteration: Independently of all other samples, run Algorithm \ref{alg:pf} with 
$N_q-N_{q-1}$ samples. Set $q=q+1$. 
If $q=p+1$ stop; otherwise go to the start of 2.}
\end{enumerate}
\caption{Approach for sampling to compute \eqref{eq:eta_0_app}, for $p\in\mathbb{Z}^+$ given.}
\label{alg:pf_est}
\end{algorithm}

\subsubsection{Approximating $[\bar{\eta}_n^l-\bar{\eta}_n^{l-1}](\varphi)$}

Let $l\in\mathbb{N}$ be given. 
We will review a method to approximate $\bar{\eta}_n^l-\bar{\eta}_n^{l-1}$ 
using coupled PFs (CPFs).
The objective of CPFs is to recursively approximate a coupling of $(\eta_n^l,\eta_n^{l-1})$. We describe the method in \cite{mlpf}.

The following exposition is from \cite{cpf_clt,mlpf}. Let $n\in\mathbb{N}$, $B\in\mathcal{X}\vee\mathcal{X}$ and $\mu\in\mathcal{P}(\mathsf{X}\times\mathsf{X})$ and define the probability measure:
\begin{eqnarray*}
\check{\Phi}_n^l(\mu)(B) & = & \mu\Big(\{F_{n-1,\mu,l} \wedge F_{n-1,\mu,l-1}\} \check{M}^l(B)\Big)
+ \Big(1-
\mu\Big(\{F_{n-1,\mu,l} \wedge F_{n-1,\mu,l-1}\}\Big)
\Big) \times \\ & & 
(\mu\otimes\mu)\Big(\Big\{\overline{F}_{n-1,\mu,l}\otimes \overline{F}_{n-1,\mu,l-1}\Big\}\bar{M}^l(B)\Big)
\end{eqnarray*}
where for $(x,x')\in\mathsf{X}\times\mathsf{X}$
\begin{eqnarray*}
\overline{F}_{n-1,\mu,l}(x,x') &  = & \frac{F_{n-1,\mu,l}(x,x')-\{F_{n-1,\mu,l}(x,x')\wedge F_{n-1,\mu,l-1}(x,x')\}}{
\mu(F_{n-1,\mu,l}-\{F_{n-1,\mu,l}\wedge F_{n-1,\mu,l-1}\})} \\
\overline{F}_{n-1,\mu,l-1}(x,x') &  = & \frac{F_{n-1,\mu,l-1}(x,x')-\{F_{n-1,\mu,l}(x,x')\wedge F_{n-1,\mu,l-1}(x,x')\}}{
\mu(F_{n-1,\mu,l-1}-\{F_{n-1,\mu,l}\wedge F_{n-1,\mu,l-1}\})} \\
F_{n-1,\mu,l}(x,x') &  = & \check{G}_{n-1,\mu,l}(x)\otimes1 \\
F_{n-1,\mu,l-1}(x,x') &  = & 1\otimes \check{G}_{n-1,\mu,l-1}(x') \\
\check{G}_{n-1,\mu,l}(x) & = & \frac{G_{n-1}(x)}{\mu(G_{n-1}\otimes 1)} \\
\check{G}_{n-1,\mu,l-1}(x') & = & \frac{G_{n-1}(x')}{\mu(1\otimes G_{n-1})} 
\end{eqnarray*}
and for $((x,x'),(z,z'))\in\mathsf{X}^2\times \mathsf{X}^2$ and $B\in\mathcal{X}\vee\mathcal{X}$
$$
\bar{M}^l(B)((x,x'),(z,z')) = \check{M}^l(B)(x,z').
$$
Now we define, recursively, for any $n\in\mathbb{N}$, $B\in\mathcal{X}\vee\mathcal{X}$
$$
\check{\eta}_n^l(B) = \check{\Phi}_n^l(\check{\eta}_{n-1}^l)(B).
$$
Note that by \cite[Proposition A.1]{mlpf} we have for $B\in\mathcal{X}$
$$
\check{\eta}_n^l(B\times\mathsf{X}) = \eta_n^l(B)\quad\textrm{and}\quad\check{\eta}_n^l(\mathsf{X}\times B) = \eta_n^{l-1}(B).
$$
The CPF at time $n$ has the following law, 
with $u_k^i=(x_k^{i,l},x_k^{i,l-1})\in\mathsf{X}\times\mathsf{X}$:
$$
\mathbb{\check{P}}^N(d(u_0^{1:N},\dots,u_n^{1:N})) = \Big(\prod_{i=1}^N \check{\eta}_0^l(du_0^i)\Big)\Big(\prod_{k=1}^n
\prod_{i=1}^N\check{\Phi}_k^l(\check{\eta}_{k-1}^{N,l})(du_k^i)\Big)
$$
where for $k\in\mathbb{N}$, 
$$
\check{\eta}_{k-1}^{N,l}(du) = \frac{1}{N}\sum_{i=1}^N \delta_{u_{k}^{i}}(du).
$$
and we set for  $s\in\{l,l-1\}$
$$
\eta_{p-1}^{N,s}(dx) = \frac{1}{N}\sum_{i=1}^N \delta_{x_{p-1}^{i,s}}(dx).
$$
To estimate $[\bar{\eta}_n^l-\bar{\eta}_n^{l-1}](\varphi)$, which will be critical in our forthcoming exposition, we have the estimate
$$
\frac{\eta_n^{N,l}(G_n\varphi)}{\eta_n^{N,l}(G_n)} - \frac{\eta_n^{N,l-1}(G_n\varphi)}{\eta_n^{N,l-1}(G_n)}.
$$
Note that this estimate is biased, but consistent, 
i.e. it converges in the limit as $N\rightarrow \infty$ 
but has a bias for any finite $N$. 
The CPF is summarized in Algorithm \ref{alg:cpf}.

\begin{algorithm}[h!]
\begin{enumerate}
\item{Initialization: For $i\in\{1,\dots,N\}$ sample $U_0^{i}$ from $\check{\eta}_0^l$. Set $n=1$.}
\item{Resampling and Sampling: For $i\in\{1,\dots,N\}$ sample $U_n^{i}$ from $\check{\Phi}_n^l(\check{\eta}_{n-1}^{N,l})$. 
Set $n=n+1$ and return to the start of 2.}
\end{enumerate}
\caption{A Coupled Particle Filter with a fixed number of samples $N\in\mathbb{Z}^+$ and a given $l\in\mathbb{N}$.}
\label{alg:cpf}
\end{algorithm}

To form our approximation of $\bar{\eta}_n^l-\bar{\eta}_n^{l-1}(\varphi)$ of $N_0$ samples, we run the CPF as described above with $N_0$ samples. To form the approximation
with $N_1$ samples, we run a CPF independently of the first CPF with $N_1-N_0$ samples and so on, for any $p\geq 2$ (i.e.~with $N_2-N_1,\dots, N_{p}-N_{p-1}$ samples).
For a given $p$, the joint law of the simulated samples is 
$$
\mathbb{\check{P}}_c(d(u_0^{1:N_p},\dots,u_n^{1:N_p}))
 = \prod_{q=0}^p \mathbb{\check{P}}^{N_q-N_{q-1}}(d(u_0^{N_{q-1}+1:N_q},\dots,u_n^{N_{q-1}+1:N_q}))
$$
and we will denote expectations w.r.t.~$\mathbb{\check{P}}_c$ as $\mathbb{\check{E}}_c$.
For $s\in\{l,l-1\}$ and any $\varphi\in\mathcal{B}_b(\mathsf{X})$ we define 
\begin{eqnarray*}
\eta_n^{N_{0:p},s}(\varphi) & := & \sum_{q=0}^p\Big(\frac{N_q-N_{q-1}}{N_p}\Big)\eta_n^{N_q-N_{q-1},s}(\varphi) \, , \\
\eta_n^{N_q-N_{q-1},s}(\varphi) & := & \frac{1}{N_q-N_{q-1}}\sum_{i=N_{q-1}+1}^{N_q}\varphi(x_n^{i,s}) \, .
\end{eqnarray*}

Finally the approximation of 
$\bar{\eta}_n^l-\bar{\eta}_n^{l-1}(\varphi)$ with $N_p$ samples is then
\begin{equation}\label{eq:eta_l_app}
[\bar{\eta}_n^{N_p,l}-\bar{\eta}_n^{N_p,l-1}](\varphi) = 
\frac{\eta_n^{N_{0:p},l}(G_n\varphi)}{\eta_n^{N_{0:p},l}(G_n)} - \frac{\eta_n^{N_{0:p},l-1}(G_n\varphi)}{\eta_n^{N_{0:p},l-1}(G_n)} \, .
\end{equation}
The procedure for sampling, in order to compute \eqref{eq:eta_l_app} is summarized in Algorithm \ref{alg:cpf_est}.

\begin{algorithm}[h!]
\begin{enumerate}
\item{Initialization: Run Algorithm \ref{alg:cpf} with $N_0$ samples. Set $q=1$. If $p=0$ stop, otherwise go to 2.}
\item{Iteration: Independently of all other samples, run Algorithm \ref{alg:cpf} with $N_q-N_{q-1}$ samples. Set $q=q+1$. If $q=p+1$ stop; otherwise go to the start of 2.}
\end{enumerate}
\caption{Approach for sampling to compute \eqref{eq:eta_l_app}, for $(p,l)\in\mathbb{Z}^+\times\mathbb{N}$ given.}
\label{alg:cpf_est}
\end{algorithm}

\subsection{Algorithm} \label{sec:algo}

Our procedure for computing unbiased estimates, based on the components developed in the previous sections, is summarized in Algorithm \ref{alg:main_method}. 
\begin{algorithm}[h!]
\begin{enumerate}
\item{For $i\in\{1,\dots,M\}$ sample $L_i\in\mathbb{Z}^+$ according to $\mathbb{P}_L$ and $P_i\in\mathbb{Z}^+$ according to $\mathbb{P}_P$. 
Denote the realizations as $(l_i, p_i)$.}
\item{If $l_i=0$ compute 
$$
\Xi_{l_i,p_i} =  \frac{1}{\mathbb{P}_P(p_i)}\{\bar{\eta}_n^{N_{p_i},0}(\varphi)-\bar{\eta}_n^{N_{p_i-1},0}(\varphi)\}
$$
where $\bar{\eta}_n^{N_{p},0}(\varphi)$ is as \eqref{eq:eta_0_app} (see Algorithm \ref{alg:pf_est}).}
\item{Otherwise, compute 
$$
\Xi_{l_i,p_i} =  \frac{1}{\mathbb{P}_P(p_i)}
\Big([\bar{\eta}_n^{N_{p_i},l_i}-\bar{\eta}_n^{N_{p_i}, l_i-1}](\varphi)  -
[\bar{\eta}_n^{N_{p_i-1},l_i}-\bar{\eta}_n^{N_{p_i-1}, l_i-1}](\varphi) 
\Big)
$$
where $[\bar{\eta}_n^{N_{p},l}-\bar{\eta}_n^{N_{p}, l-1}](\varphi) $ is as \eqref{eq:eta_l_app} (see Algorithm \ref{alg:cpf_est}).}
\item{Return the estimate:
\begin{equation}\label{eq:ub_pf_est}
\widehat{\bar{\eta}_n(\varphi)} = \frac{1}{M}\sum_{i=1}^M \frac{1}{\mathbb{P}_L(l_i)}\Xi_{l_i,p_i}.
\end{equation}
}
\end{enumerate}
\caption{Algorithm for Unbiased Estimation of $\bar{\eta}_n(\varphi)$.}
\label{alg:main_method}
\end{algorithm}

A few remarks can help to clarify the algorithm.
\begin{itemize}
\item The terms in the differences in $\Xi_{l,p}$ are not independent; they will share $N_{p-1}$ common samples that have been produced by the PF/CPF.
\item 
Each sample in the estimate \eqref{eq:ub_pf_est} can be computed in parallel; i.e.~this is amenable to parallel computation. 
\item 
One can correlate them $P$ and $L$. There is no reason why 
they need to be independent random variables. 
\item 
The algorithm is online,
i.e. the computational cost per observation time
is fixed. For each sample in the estimate \eqref{eq:ub_pf_est},
one can simply fix the $L_i$ and $P_i$ sampled at time $0$ and update the estimates of the filter as time progresses, by using the sequential nature of the PF/CPF algorithms.
\item 
At this stage, we have still not established that the estimator
is unbiased with finite variance, nor have we investigated the associated computational effort to compute the estimate; this is the topic of Section \ref{sec:theory}.
\end{itemize}

The scheme that has been proposed is a type of double randomization (or double `Rhee \& Glynn', following from the work \cite{rhee}) where one randomizes twice; firstly with respect to the discretization level of the diffusion and secondly to obtain unbiased estimates of the increments. The first randomization seems necessary, given the current state-of-the-art of stochastic computation; we are assuming that unbiased simulation methods (e.g.~\cite{fearn} and the references therein) are not sensible in our problems of interest. 
For the second randomization, there are several alternatives which could be considered. The first is to replace the 'single-term' estimator that we are currently using with the `coupled-sum' estimator (\cite{rhee}). 
Methodologically, this is not significantly different from what we have suggested, but the conditions for unbiasedness and finite variance change marginally. More precisely, one
would use the term
$$
\Xi_{l_i,p_i} = \sum_{s=0}^{p_i}\frac{1}{\sum_{q=s}^{\infty} \mathbb{P}_P(q)}
\Big([\bar{\eta}_n^{N_{s},l_i}-\bar{\eta}_n^{N_{s}, l_i-1}](\varphi)  -
[\bar{\eta}_n^{N_{s-1},l_i}-\bar{\eta}_n^{N_{s-1}, l_i-1}](\varphi)\Big) 
$$
in 3.~of Algorithm \ref{alg:main_method}, with a similar type expression in 2.~of Algorithm \ref{alg:main_method}.
A second alternative would be to use an unbiased sampling scheme 
based upon Markov chain simulation (e.g.~\cite{glynn2, jacob2}). Although these latter schemes would have to be modified and enhanced to be applicable in the context here, the main issue with applying them is that they are not intrinsically `online'. 
We note also that these schemes themselves are based upon randomization methods, and hence one would be using a double randomization again.
We also remark that the approach to obtain $\Xi_{l,p}$, which is essentially a randomization on the number of samples, is similar to the approach in \cite{glynn}. In \cite{glynn} the authors also use an associated idea to unbiasedly estimate non-linear functions of expectations. The approach is related, except they rely on the using independent samples from a probability of interest; in this scenario one can use all the same 
$N_p$ samples in $\Xi_{l,p}$ to construct both the fine and course approximations,
whereas, this does not seem to be possible when the samples are not independent. 
This is related to the antithetic coupling described in \cite{giles2}.

One may attempt to construct an estimator using a single randomization.
It is not clear how to construct an efficient method with this approach. 
In Algorithm \ref{alg:sl_main_method} we present a potential single randomization strategy and
the estimator is given in \eqref{eq:ub_pf_est_sl}. 
The framework for this estimator is just a 
single-term estimator as discussed (for instance) in Section \ref{sec:gen} and,
as described there, one can establish that \eqref{eq:ub_pf_est_sl} is both unbiased and of finite variance if \eqref{eq:ub_fv_cond} is satisfied. 
In Section \ref{subsec:singlerand} we will explain why 
this estimator does not work well.

\begin{algorithm}[h!]
\begin{enumerate}
\item{For $i\in\{1,\dots,M\}$ sample $L_i\in\mathbb{Z}^+$ according to $\mathbb{P}_L$.
Let $l_i$ denote the realizations.}
\item{If $l_i=0$ compute 
$$
\Xi_{l_i} = \frac{1}{\mathbb{P}_L(l_i)}\bar{\eta}_n^{N_{0},0}(\varphi)
$$
where $\bar{\eta}_n^{N_{0},0}(\varphi)$ is as \eqref{eq:eta_0_app} (see Algorithm \ref{alg:pf_est}) with $p=0$.}
\item{Otherwise, compute 
$$
\Xi_{l_i} = \frac{1}{\mathbb{P}_L(l_i)}
\Bigg( \Big(\frac{N_{l_i}-N_{l_i-1}}{N_{l_i}}\Big ) \bar{\eta}_n^{N_{l_i}-N_{l_i-1},l_i}(\varphi) + 
\Big [\Big(\frac{N_{l_i-1}}{N_{l_i}}\Big) \bar{\eta}_n^{N_{l_i-1},l_i}-\bar{\eta}_n^{N_{l_i-1}, l_i-1}
\Big ](\varphi) \Bigg)
$$
where $\bar{\eta}_n^{N_{l_i}-N_{l_i-1},l_i}(\varphi)$ is computed using Algorithm \ref{alg:pf}, independently of $[\bar{\eta}_n^{N_{l_i-1},l_i}-\bar{\eta}_n^{N_{l_i-1}, l_i-1}](\varphi)$ which is as \eqref{eq:eta_l_app} (see Algorithm \ref{alg:cpf_est}).}
\item{Return the estimate:
\begin{equation}\label{eq:ub_pf_est_sl}
\widehat{\bar{\eta}_n(\varphi)} = \frac{1}{M}\sum_{i=1}^M \frac{1}{\mathbb{P}_L(l_i)}\Xi_{l_i}.
\end{equation}
}
\end{enumerate}
\caption{A Single Randomized Algorithm for Unbiased Estimation of $\bar{\eta}_n(\varphi)$.}
\label{alg:sl_main_method}
\end{algorithm}

\section{Theoretical Results}\label{sec:theory}

\subsection{Unbiasedness and Finite Variance}

Our objective is now to establish that the estimator \eqref{eq:ub_pf_est} is unbiased with finite variance. To show this, we must show
that there exist positive probability mass functions $\mathbb{P}_L$, $\mathbb{P}_P$ 
and $(N_p)_{p\in\mathbb{Z}^+}$ an increasing sequence of positive integers with $\lim_{p\rightarrow\infty}N_p=\infty$ 
such that first \eqref{eq:xipl_cond} holds and then
\eqref{eq:xi_cond} also holds, under the particular strategy detailed in Algorithm \ref{alg:main_method}. We first state two results that can help to achieve our
objectives.

We begin with the PF and the proof of this result can be found in Appendix \ref{app:pf_prf}.
\begin{prop}\label{prop:pf_prop}
For any $n\in\mathbb{Z}^+$ there exists a $C<+\infty$ such that for any $p\in\mathbb{Z}^+$, $N_p>N_{p-1}>\cdots>N_0\geq 1$, $\varphi\in\mathcal{B}_b(\mathsf{X})$:
$$
\mathbb{E}_c\Big[\Big(\bar{\eta}_n^{N_p,0}(\varphi)-\bar{\eta}_n^0(\varphi)\Big)^2\Big] \leq \frac{C\|\varphi\|_{\infty}^2}{N_p}\Big(1+\frac{p^2}{N_p}\Big).
$$
\end{prop}

We introduce the following assumptions, which will be needed for the case of the CPF.
\begin{hypA}\label{hyp:1}
There exist $c,C<+\infty$ such that for every $n\in\mathbb{Z}^+$ we have 
\begin{itemize}
\item[{\rm (i)}] {\bf boundedness}: $c^{-1} < G_n(x) < c$ for all $x \in \mathsf{X}$; 
\item[{\rm (ii)}] {\bf a globally Lipschitz condition}: $|G_n(x) - G_n(x')| \leq C |x-x'|$, for all $(x,x') \in \mathsf{X}\times\mathsf{X}$ and $|\cdot|$ is the $L_2-$norm.
\end{itemize}
\end{hypA}

\begin{hypA}\label{hyp:2}
There exists a  $C<+\infty$ such that for each $(x,x')\in \mathsf{X}\times\mathsf{X}$, $l\in\mathbb{Z}^+$ and $\varphi\in\mathcal{B}_b(\mathsf{X})\cap\textrm{Lip}(\mathsf{X})$
$$
|M^l(\varphi)(x) - M^l(\varphi)(x')| \leq C\|\varphi\|_{\infty}~|x-x'|.
$$
\end{hypA}

The proof of the following result for the CPF is in Appendix \ref{app:cpf_prf}.
\begin{theorem}\label{theo:cpf_res}
Assume (A\ref{hyp:1}-\ref{hyp:2}). Then for any $n\in\mathbb{Z}^+$ there exists a $C<+\infty$ such that for any $(l,p)\in\mathbb{N}\times\mathbb{Z}^+$, $N_p>N_{p-1}>\cdots>N_0\geq 1$, $\varphi\in\mathcal{B}_b(\mathsf{X})\cap\textrm{\emph{Lip}}(\mathsf{X})$:
\begin{equation}\label{eq:cpf_bound}
\mathbb{\check{E}}_c\Big[\Big(
[\bar{\eta}_n^{N_p,l}-\bar{\eta}_n^{N_p,l-1}](\varphi)
-[\bar{\eta}_n^l-\bar{\eta}_n^{l-1}](\varphi)\Big)^2\Big] \leq 
\frac{C\Delta_l^{\beta}\|\varphi\|_{\infty}^2}{N_p}\Big(1+\frac{p^2}{N_p}\Big)
\end{equation}
where $\beta=\frac{1}{2}$ if $b$ is non-constant and $\beta=1$ if $b$ is constant.
\end{theorem}

\begin{theorem}\label{theo:main_res}
Assume (A\ref{hyp:1}-\ref{hyp:2}). 
Let $(n,\varphi)\in\mathbb{Z}^+\times(\mathcal{B}_b(\mathsf{X})\cap\textrm{\emph{Lip}}(\mathsf{X}))$.
Then there exist choices of positive probability mass functions $\mathbb{P}_L$, $\mathbb{P}_P$ 
and $(N_p)_{p\in\mathbb{Z}^+}$ an increasing sequence of integers with $\lim_{p\rightarrow\infty}N_p=\infty$ such that  \eqref{eq:xipl_cond} and
\eqref{eq:ub_fv_cond} hold and hence that each summand in \eqref{eq:ub_pf_est} is an unbiased estimator of $\bar{\eta}_n(\varphi)$ with finite variance.
\end{theorem}

\begin{proof}
Throughout the proof $C$ is a constant that does not depend on $l$ nor $p$ but whose value may change from line to line. 
In the case of \eqref{eq:xipl_cond}, we have, if $l=0$ and any $p\in\mathbb{Z}^+$ by Proposition \ref{prop:pf_prop},
\begin{equation}\label{eq:main_res1}
\mathbb{E}[\Xi_{l,p}^2] \leq \frac{1}{\mathbb{P}_P(p)^2}\frac{C\|\varphi\|_{\infty}^2}{N_{p-1}\vee N_0}\Big(1+\frac{p^2}{N_{p-1}\vee N_0}\Big).
\end{equation}
Then, for instance, setting 
$N_p=2^p$ and $\mathbb{P}_P(p)\propto 2^{-p}(p+1)\log_2(p+2)^2$ ensures that, the R.H.S.~of the displayed equation multiplied by $\mathbb{P}_P(p)$ is summable over $p\in\mathbb{Z}^+$; this verifies \eqref{eq:xipl_cond} when $l=0$. 
Now, if $l\in\mathbb{N}$, we have by Theorem \ref{theo:cpf_res}
\begin{equation}\label{eq:main_res2}
\mathbb{E}[\Xi_{l,p}^2] \leq \frac{1}{\mathbb{P}_P(p)^2}\frac{C\Delta_l^{\beta}\|\varphi\|_{\infty}^2}{N_{p-1}\vee N_0}\Big(1+\frac{p^2}{N_{p-1}\vee N_0}\Big)
\end{equation}
and thus by the above argument, \eqref{eq:xipl_cond} is verified when $l\in\mathbb{N}$.  Thus \eqref{eq:xi_cond} holds with our choice of $\Xi_{l,p}$.

To verify that \eqref{eq:ub_fv_cond} holds, we have
\begin{eqnarray*}
\sum_{l\geq 0}\frac{1}{\mathbb{P}_L(l)}\mathbb{E}[\Xi_l^2] & = & \sum_{l\in\mathbb{Z}^+}\frac{1}{\mathbb{P}_L(l)}\sum_{p\in\mathbb{Z}^+} 
\mathbb{P}_P(p) \mathbb{E}[\Xi_{l,p}^2] \\
& \leq & C\|\varphi\|_{\infty}^2\sum_{l\in\mathbb{Z}^+}\frac{1}{\mathbb{P}_L(l)}\sum_{p\in\mathbb{Z}^+} 
\frac{1}{\mathbb{P}_P(p)}\frac{\Delta_l^{\beta}}{N_{p-1}\vee N_0}\Big(1+\frac{p^2}{N_{p-1}\vee N_0}\Big)
\end{eqnarray*}
where we have applied \eqref{eq:main_res1}-\eqref{eq:main_res2} (note that $\Delta_0$ is $\mathcal{O}(1)$). 
Setting, for example $N_p=2^p$, $\mathbb{P}_P(p)\propto 2^{-p}(p+1)\log_2(p+2)^2$ 
and $\mathbb{P}_L(l)\propto (\Delta_l^{\beta})^{\rho}$ for any $\rho\in(0,1)$
ensures that \eqref{eq:ub_fv_cond} holds and hence that the proof is completed.
\end{proof}

\begin{rem}
In our proof, we have not used the fact that 
$\bar{\eta}_n^{N_{p_i},0}(\varphi)-\bar{\eta}_n^{N_{p_i-1},0}(\varphi)$ 
uses common samples
in $\bar{\eta}_n^{N_{p_i},0}(\varphi)$ and $\bar{\eta}_n^{N_{p_i-1},0}(\varphi)$. 
However, one can check that the fact that there are $N_{p_i}-N_{p_i-1}$ independent
and additional samples in the estimate $\bar{\eta}_n^{N_{p_i},0}(\varphi)$ means that there is not a substantial improvement in the bounds when incorporating these common samples
into computing the upper-bound.
\end{rem}

\subsection{Cost}

On inspection of the proof of Theorem \ref{theo:main_res} we needed to choose $\mathbb{P}_L$ and $\mathbb{P}_P$  and $(N_p)_{p\in\mathbb{Z}^+}$, so that
\begin{eqnarray*}
\sum_{p\in\mathbb{Z}^+} \frac{1}{\mathbb{P}_P(p)}\frac{1}{N_{p-1}\vee N_0}\Big(1+\frac{p^2}{N_{p-1}\vee N_0}\Big) & < & \infty\\
\sum_{l\in\mathbb{Z}^+}\frac{\Delta_l^{\beta}}{\mathbb{P}_L(l)} & < & \infty.
\end{eqnarray*} 
The expected cost of producing a single sample of the estimate \eqref{eq:ub_pf_est} is $\sum_{l\in\mathbb{Z}^+}\sum_{p\in\mathbb{Z}^+}\mathbb{P}_L(l)\mathbb{P}_P(p)
\Delta_l^{-1} N_p$. As a result it is unlikely that one can select $\mathbb{P}_L$ and $\mathbb{P}_P$  and $(N_p)_{p\in\mathbb{Z}^+}$ so that the estimator \eqref {eq:ub_pf_est} is
unbiased and of finite variance, but also of finite expected cost. So 
our estimate is in the sub-canonical regime of \cite{rhee}.

In our discussion, the cost will not consider the impact of the time parameter $n$ as our bounds in Proposition \ref{prop:pf_prop} and Theorem \ref{theo:cpf_res} have constants that grow exponentially with $n$.
As considered in \cite{cpf_clt}, we expect that the bounds can be made uniform in $n$, with a substantial increase in technical difficulty. 
Suppose that the diffusion coefficient $b$ is constant, so that $\beta=1$ and we set 
$N_p=N_0 2^p$, $\Delta_l=2^{-l}$, 
$\mathbb{P}_P(p)\propto N_p^{-1}(p+1)\log_2(p+2)^2$ and $\mathbb{P}_L(l)\propto 
\Delta_l^{-1}(l+1)\log_2(l+2)^2$. 
Then it easily follows that \eqref{eq:ub_pf_est} is an unbiased estimator of 
$\bar{\eta}_n(\varphi)$ with finite variance. 
Moreover, if one sets
$M=\mathcal{O}(\epsilon^{-2})$ (with $\epsilon>0$ arbitrary), so that the variance is $\mathcal{O}(\epsilon^{2})$, one can follow the analysis of \cite[pp.~1035]{rhee}, with some additional calculations, to establish that the cost to achieve this variance is $\mathcal{O}(\epsilon^{-2}|\log(\epsilon)|^{2+\delta})$ for any $\delta>0$. If one compares
to the methodology of \cite{mlpf} (the MLPF), as mentioned previously, to obtain a MSE of  $\mathcal{O}(\epsilon^{2})$, the cost is $\mathcal{O}(\epsilon^{-2}|\log(\epsilon)|^{2})$.
Therefore, unbiased estimator has a cost which is slightly larger than that of the MLPF.
We remark that the costs (both for the unbiased method and the MLPF) are determined
by the value of $\beta$, and that for the CPF adopted $\beta$ is half of the forward rate. These rates can be improved by the CPF in \cite{mlpf_new}, although there are at present no finite sample proofs about that technique.

\subsubsection{Estimator with a single randomization}
\label{subsec:singlerand}

Before moving to the numerical experiments, we briefly 
analyze the cost of Algorithm \ref{alg:sl_main_method}. 
Using the analysis above 
one can establish that for $\Xi_l$ as in Algorithm \ref{alg:sl_main_method}, for any $n\in\mathbb{Z}^+$, there exists a $C<+\infty$ such that for any 
$l\in\mathbb{N}$, $N_l>N_{l-1}>\cdots>N_0\geq 1$ and $\varphi\in\mathcal{B}_b(\mathsf{X})\cap\textrm{Lip}(\mathsf{X})$;
\begin{equation}\label{eq:xil_sing_est}
\mathbb{E}[\Xi_l^2] \leq C\|\varphi\|_{\infty}^2\Big(\frac{1}{N_l-N_{l-1}}+\frac{\Delta_l^{\beta}}{N_{l-1}}\Big(1+\frac{(l-1)^2}{N_{l-1}}\Big)\Big)
\end{equation}
where $\beta=\frac{1}{2}$ if $b$ is non-constant and $\beta=1$ if $b$ is constant. The expected cost of computing \eqref{eq:ub_pf_est_sl} is $\sum_{l\in\mathbb{Z}^+}\mathbb{P}_L(l)
2^l N_l$. One can check, that given the upper-bound in \eqref{eq:xil_sing_est} and the condition \eqref{eq:ub_fv_cond}, it is unlikely that one can find a $\mathbb{P}_L$ and an increasing sequence $(N_l)_{l\in\mathbb{Z}^+}$  so that the estimate is simultaneously unbiased with finite variance and has finite expected cost; this is again the sub-canonical regime of \cite{rhee}. If one chooses $N_l=N_0 2^l$ and, as in \cite[pp.~1035]{rhee} 
$\mathbb{P}_L(l)\propto 2^{-l} (l+1)\log_2(l+2)^2$, then one can show that to achieve a variance
of $\mathcal{O}(\epsilon^2)$ (for $\epsilon>0$ arbitrary), the order of the work is $\mathcal{O}(\epsilon^{-4}|\log(\epsilon)|^{4+\delta})$ for any $\delta >0$. This is extremely poor.
For instance, if the diffusion coefficient $b(X_t)$ is non-constant, 
then \cite{mlpf} show that the method there can achieve a mean square error of 
$\mathcal{O}(\epsilon^2)$ for a work of $\mathcal{O}(\epsilon^{-2.5})$, 
under suitable assumptions. 
As a result, we have decided not to use single randomization approaches here.

\section{Numerical Results}\label{sec:numerics}

\subsection{Model Settings} 
The numerical performance of our unbiased estimator (9) will be compared with that of MLPF (see \cite{mlpf}), with four examples of diffusions considered in this paper. Recall that the diffusions take the following form
$$
	dZ_{t} = a(Z_{t})dt + b(Z_{t})dW_{t}, \,\,\,\,\,\,\,\,\,\,Z_{0} = x^{*} 
$$
with $Z_{t} \in \mathbb{R}^{d}, t \geq 0,$ and $\{W_{t}\}_{t\geq0}$ a Brownian motion of appropriate dimension. 
We also set $(X_0,X_1,\dots)$ a discrete time skeleton of the process $\{Z_t\}_{t\geq 0}$ at lag 1 times i.e.~$X_k=Z_{k+1}$, $k\in\{0,1,\dots\}$. 
In addition, data $(y_{1},...\,,y_{n})$ are available with $Y_{k}$ obtained at time $k\in\{1,2,\dots\}$, and  $Y_{k}|X_{k-1}$ has density function $G(x_{k-1},y_{k})$. The objective is the estimation of expectations (the function is denoted $\varphi:\mathbb{R}^d\rightarrow\mathbb{R}$) w.r.t.~the filter.

To obtain a data set $(y_1,\dots,y_n)$, $n=100$, we either generate a signal from the diffusion (if possible) or an Euler discretization of the diffusion at level 9 and then
generate data from the density function $G(x_{k-1},y_k)$. Below are the detailed settings of the four diffusions models we will be considering in our simulation; throughout $d=1$ (the dimension of the hidden diffusion).

	 
 \textbf{Ornstein-Uhlenbeck Process}\,\,\,\,\,First, consider the following OU process,
\begin{equation*}
dZ_{t} = \theta(\mu-Z_{t})dt + dW_{t}, \\
\end{equation*}
\begin{equation*}
Y_{k}|X_{k-1} \sim \mathcal{N}(x_{k-1},\tau^{2}), \,\,\,\,\, \varphi(x) = x \, .
\end{equation*}
The exact value of the first moment of the filter can be computed using a Kalman filter. The constants in the example are, $Z_{0}=0$, $\theta=1$, $\mu=0$, $\tau^{2}=0.2$. \\

\textbf{Geometric Brownian Motion}\,\,\,\,\, Next consider the GBM process,
\begin{equation*}
dZ_{t} = \mu Z_{t}dt + \sigma X_{t} dW_{t} \\
\end{equation*}
\begin{equation*}
Y_{k}|X_{k-1} \sim \mathcal{N}(\log(x_{k-1}),\tau^{2}), \,\,\,\,\, \varphi(x) = x \, .
\end{equation*}
The transition density of the diffusion is available analytically. The constants are, $Z_{0} = 1$, $\tau^{2}=0.01$, $\sigma=0.2$ and $\mu = 0.02$. \\

\textbf{Langevin Stochastic Differential Equation}\,\,\,\,\, Here the SDE is given by
\begin{equation*}
dZ_{t} = \frac{1}{2}\nabla \log\pi(Z_{t})dt + dW_{t} \\
\end{equation*}
\begin{equation*}
Y_{k}|X_{k-1} \sim \mathcal{N}(0,e^{x_{k-1}}), \,\,\,\,\, \varphi(x) = x
\end{equation*}
where $\pi(x)$ denotes a probability density function. The density $\pi(x)$ is chosen as the Student's t-distribution with degrees of freedom \emph{v} = 10. Initial value $Z_{0}=0$.  \\

\textbf{An SDE with a Non-Linear Diffusion Term}\,\,\,\,\, Last, the following SDE is considered,
\begin{equation*}
dZ_{t} = \theta(\mu-Z_{t})dt + \frac{1}{\sqrt{1+Z_{t}^{2}}} dW_{t} \\
\end{equation*}
\begin{equation*}
Y_{k}|X_{k-1} \sim \mathcal{L}(\log(x_{k-1}),s), \,\,\,\,\, \varphi(x) = x
\end{equation*}
The constants are $Z_{0}=0$, $\theta=1$, $\mu=0$, and $s=\sqrt{0.1}$. This example is abbreviated NLD in the remainder of this section.   \\

\subsection{Simulation Settings}
 In our simulation, we applied Algorithm \ref{alg:main_method} to obtain the unbiased estimator 
 $\frac{1}{M}\sum_{i=1}^M \frac{1}{\mathbb{P}_L(l_i)}\Xi_{l_i,p_i}$.
 We use the Wasserstein coupled resampling method from \cite{mlpf_new} (see also \cite{cpf_clt}) 
 in place of Algorithm \ref{alg:cpf} to get the inner increment, which is expected (but not proven) 
 to yield the improved rate $\beta=2$ (resp.~$\beta=1$), as in Theorem \ref{theo:cpf_res},  for constant (resp.~non-constant) diffusion coefficients.
 As a comparison, we will simulate the MLPF algorithm (implementation of the MLPF algorithm is detailed in \cite[Section 5]{mlpf}) with the Wasserstein resampling method.
If the true value of the filter is not available, we will use a particle filter at level 13 (or using the exact diffusion dynamics, if available), with a large number of particles ($10^5$, repeated 100 times) to approximate its value - this will be the reference solution we use in our simulations.
 
 
The MLPF method will induce a bias when estimating the filter, which we denote $B_L$ and a variance $V_L$, where $L$ is the chosen level of discretization of the diffusion process.
In the MLPF method
 one must choose the number of samples used to approximate the differences of the filters at levels $l$ and $l-1$, denoted $M_l$. Our target MSE will be 
 $\mathcal{O}(2^{-2L})$. In the case that the diffusion coefficient is constant (resp.~non-constant) we set $M_{l} = \mathcal{O}(2^{2L-1.5l})$ (resp.~$M_{l} = \mathcal{O}(2^{2L-l}L)$).
  In practice we assume MSE $= C_{0}\,2^{-2L} = B_{L}^2+V_{L}$, we fit the constant $C_{0}$ by using the simulation results of the MLPF algorithm up-to discretization level $L$, where we set $M_{l} = C_{1}2^{2L-l}L$ (resp. ~$M_l=C_{1}2^{2L-1.5l}$) for diffusions with constant (resp.~non-constant) diffusion coefficients. The constant $C_{1}$ is tuned so that $B_{L}$ and $V_{L}$ is balanced and is roughly equal to each other.  These latter quantities are estimated by repeating the MLPF algorithm 100 times. 

  For the unbiased estimator, we have to truncate the values of $P$ and $L$ in practice, since huge values of either random variable cannot be feasibly processed in a
  reasonable amount of time. To obtain an MSE $=C_{0}2^{-2L}$ for the unbiased estimator, we choose an $L_{\textrm{max}}$ (the maximum value of $L$) value such that MSE level of the unbiased estimator can drop below $C_{0}2^{-2L}$. To specify the joint distribution of $P$ and $L$ we will allow  $(L,P)\in\{0,1,\dots,L_{\textrm{max}}\}\times \{0,1,\dots,P_{\textrm{max}}\}$ and detail $\mathbb{P}_L(l)$ and then $\mathbb{P}_{P|l}(p|l)$.
  We set $\mathbb{P}_L(l)\propto 2^{-1.5l}\mathbb{I}_{\{0,1,\dots,L_{\textrm{max}}\}}(l)$. Then 
  $$
  \mathbb{P}_{P|l}(p|l) \propto
    \left\{\begin{array}{ll}
    2^{4-p} & \textrm{if}~p\in\{0,1,\dots, 4\wedge(L_{\textrm{max}}-l)\} \\
    2^{-p}p[\log_{2}(p)]^2& \textrm{if}~p\in\{5,6,\dots, (L_{\textrm{max}}-l)\}~\textrm{and}~(L_{\textrm{max}}-l)\geq 5 \\
    0 & \textrm{otherwise}
     \end{array}\right.
  $$
  The distribution is chosen in this way so that the induced bias (i.e.~lowest achievable MSE)
  is comparable to the MLPF method and such that the cost to achieve the target MSE is 
  again comparable. Of course the estimator is no longer unbiased, however 
  if one can choose the target MSE ahead of time then such appropriately chosen bias 
  is inconsequential for the ultimate estimator \eqref{eq:ub_pf_est}, 
  and the cost to obtain this estimator 
  is lower than the genuinely unbiased one. 
   The samples of \eqref{eq:ub_pf_est} are still i.i.d.~and so the method can 
  still be easily parallelized. The value of $N_0$ for the choice of $(N_p)_{p\in\mathbb{Z}^+}$ ($N_p=N_0 2^p$) is 10 for the case of a constant diffusion coefficient and 50 in the non-constant case.
%

  The MSE  of our (truncated) unbiased estimator is $C_{2} 2^{-2L_{\textrm{max}}} + C_{3}M^{-1}$ and we use a few $L_{\textrm{max}}$ values to run the unbiased estimator, to obtain approximate bias and variance values and thus estimated values of $C_{2}$ and $C_{3}$. Based upon these values we can find the appropriate values of $L_{\textrm{max}}$ and $M$ such that a given MSE level is obtained, with balanced square bias and variance; this means we should have that $C_{2}2^{-2 L_{\textrm{max}}} \approx C_{3}M^{-1}$. In the simulation section, we fit only $C_{2}$ such that $C_{2} 2^{-2 L_{\textrm{max}}}$ is smaller than the target MSE level, then we increase $M$ until the approximate MSE of the unbiased estimator hits the target. This is possible because the variance goes to zero as $M$ goes to infinity. Since the unbiased estimator consists of $M$ i.i.d.~realizations of $\frac{1}{\mathbb{P}_L(l_i)}\Xi_{l_i,p_i}$, after choosing the value of $L_{\textrm{max}}$ based on a given MSE, we can simulate a large number (say $10^6$) of i.i.d.~$\frac{1}{\mathbb{P}_L(l_i)}\Xi_{l_i,p_i}$, then for a certain $M$ we compute (say) 100 estimates (with no overlap of $\frac{1}{\mathbb{P}_L(l_i)}\Xi_{l_i,p_i}$) so as to estimate the bias and variance in our results.
  
    

   The aim is to compare the two estimators cost at same MSE level. 
   We first simulate the MLPF algorithm with $L \in \{1,2,3,4,5,6,7\}$ and obtain the respective MSE and cost values. Then we apply the unbiased estimator to obtain the same MSE levels and record its cost value.  As we mentioned in the last paragraph, given a proper $L_{\textrm{max}}$ we can simulate a large number of realizations and then obtain respective MSE values for different $M$ values (which corresponds to different cost values). Consequently, we can then interpolate the MSE values onto a uniform cost grid. In the simulation, we choose the $L_{\textrm{max}}$ such that MSE for the unbiased estimator can drop below the MSE for the MLPF algorithm with $L=7$, then obtain an interpolated plot of MSE against cost. From the MSE-cost graph, we extract the cost required for the unbiased estimator to obtain a matching MSE to that of the MLPF, and this will allow us to compare their cost.



\subsection{Simulation Results}

\begin{figure}[h]\centering
\subfigure[OU]{\includegraphics[width=8cm,height=6cm]{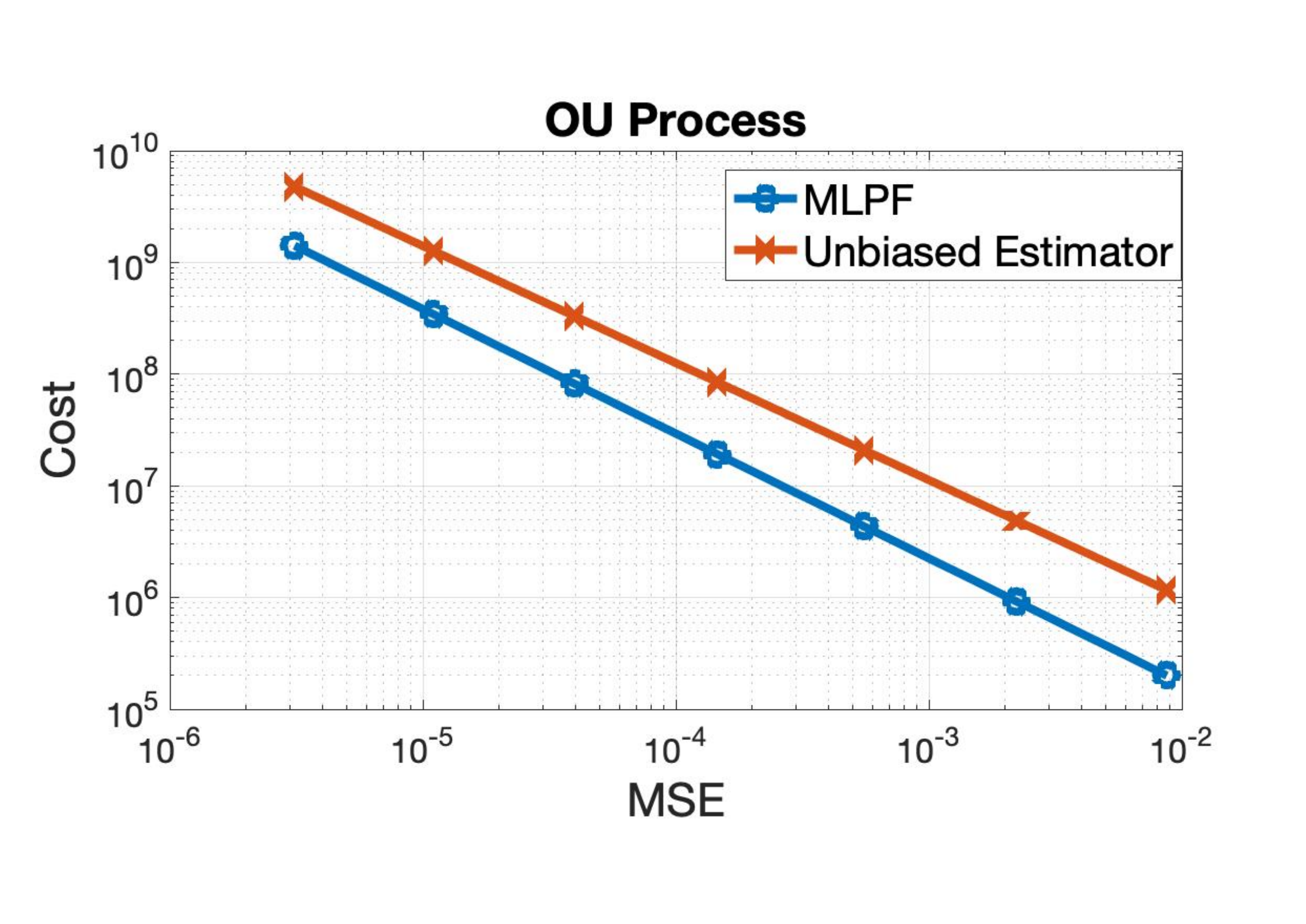}}
\subfigure[Langevin]{\includegraphics[width=8cm,height=6cm]{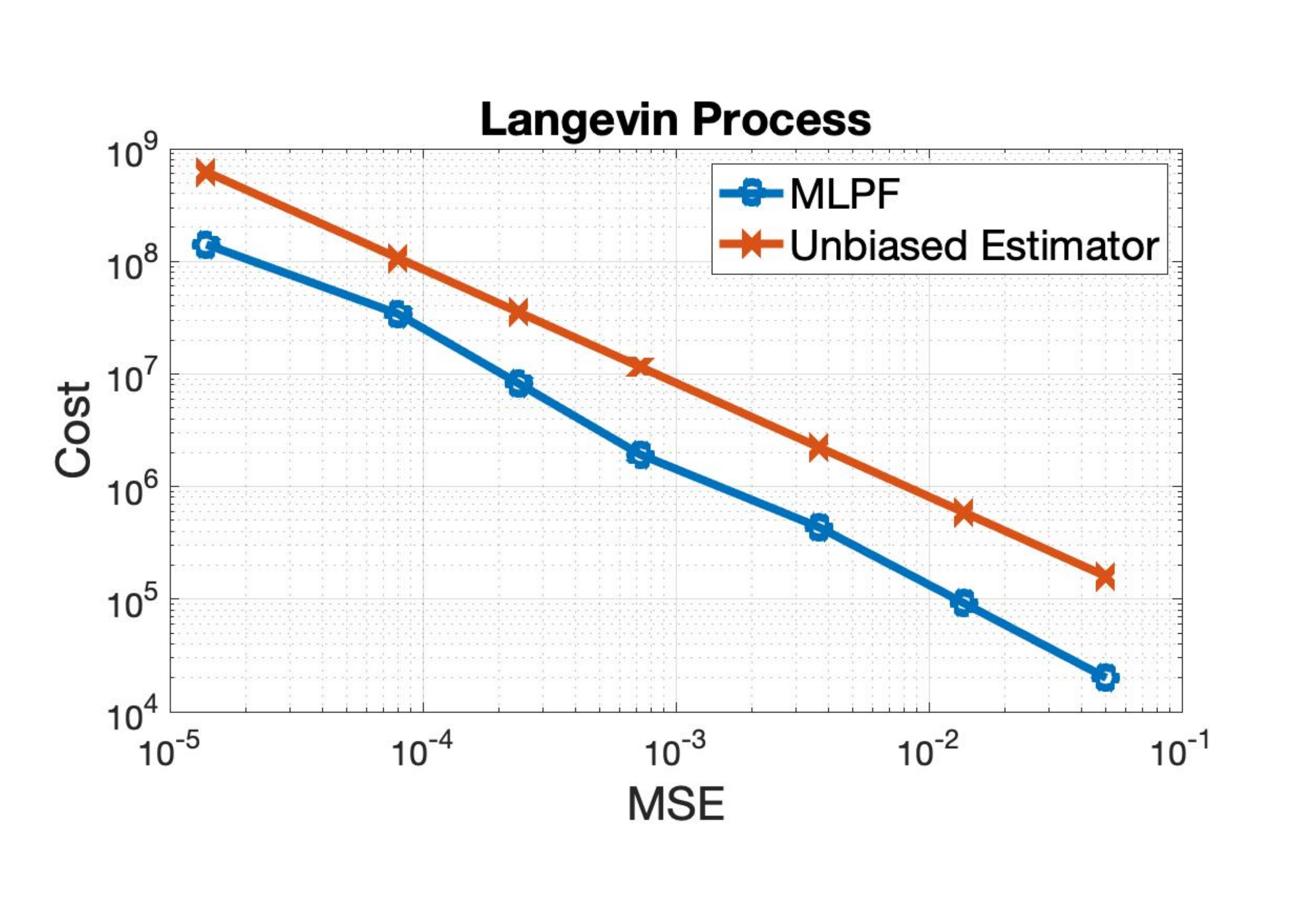}}
\subfigure[Non-Linear]{\includegraphics[width=8cm,height=6cm]{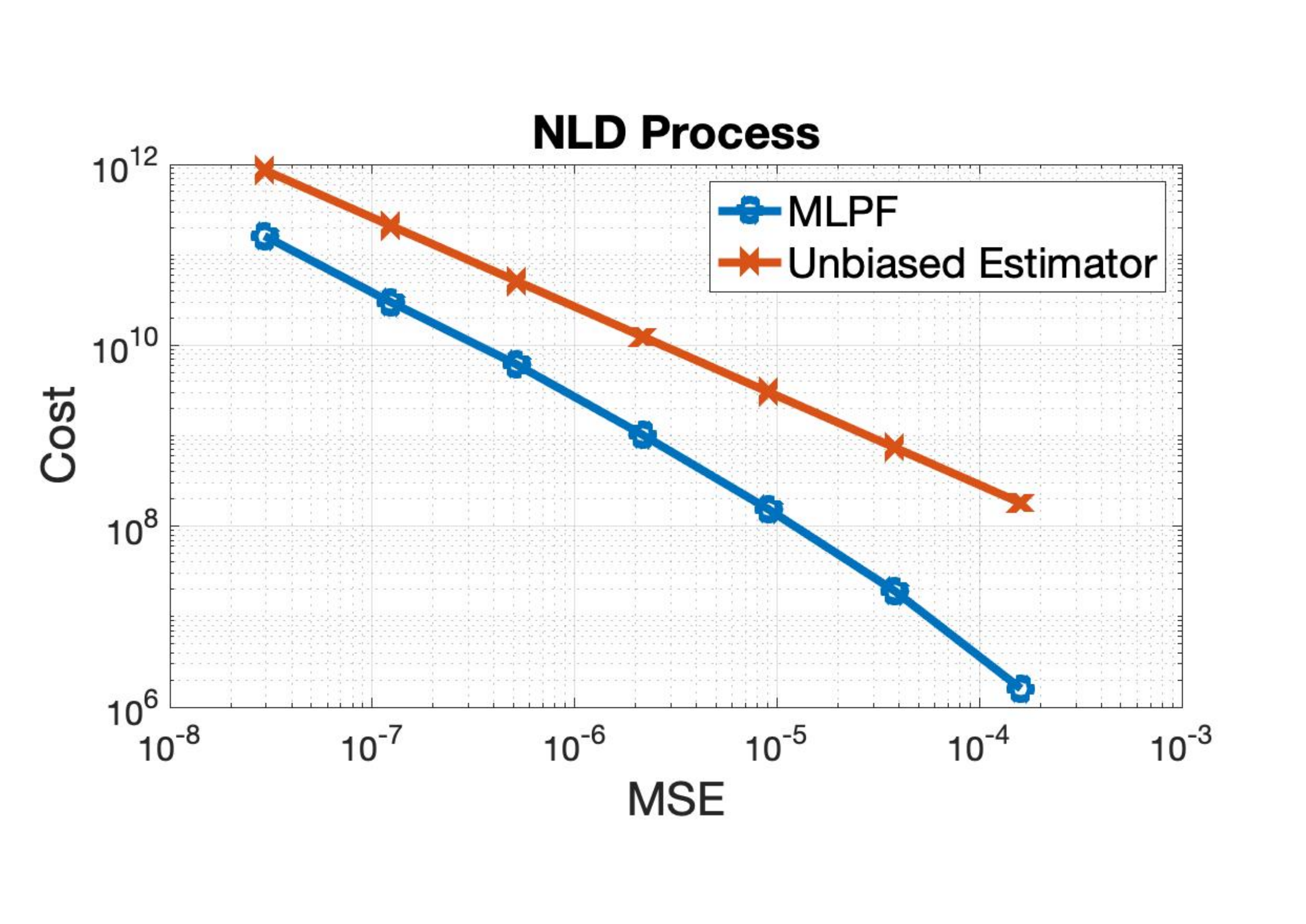}}
\subfigure[Geometric Brownian Motion]{\includegraphics[width=8cm,height=6cm]{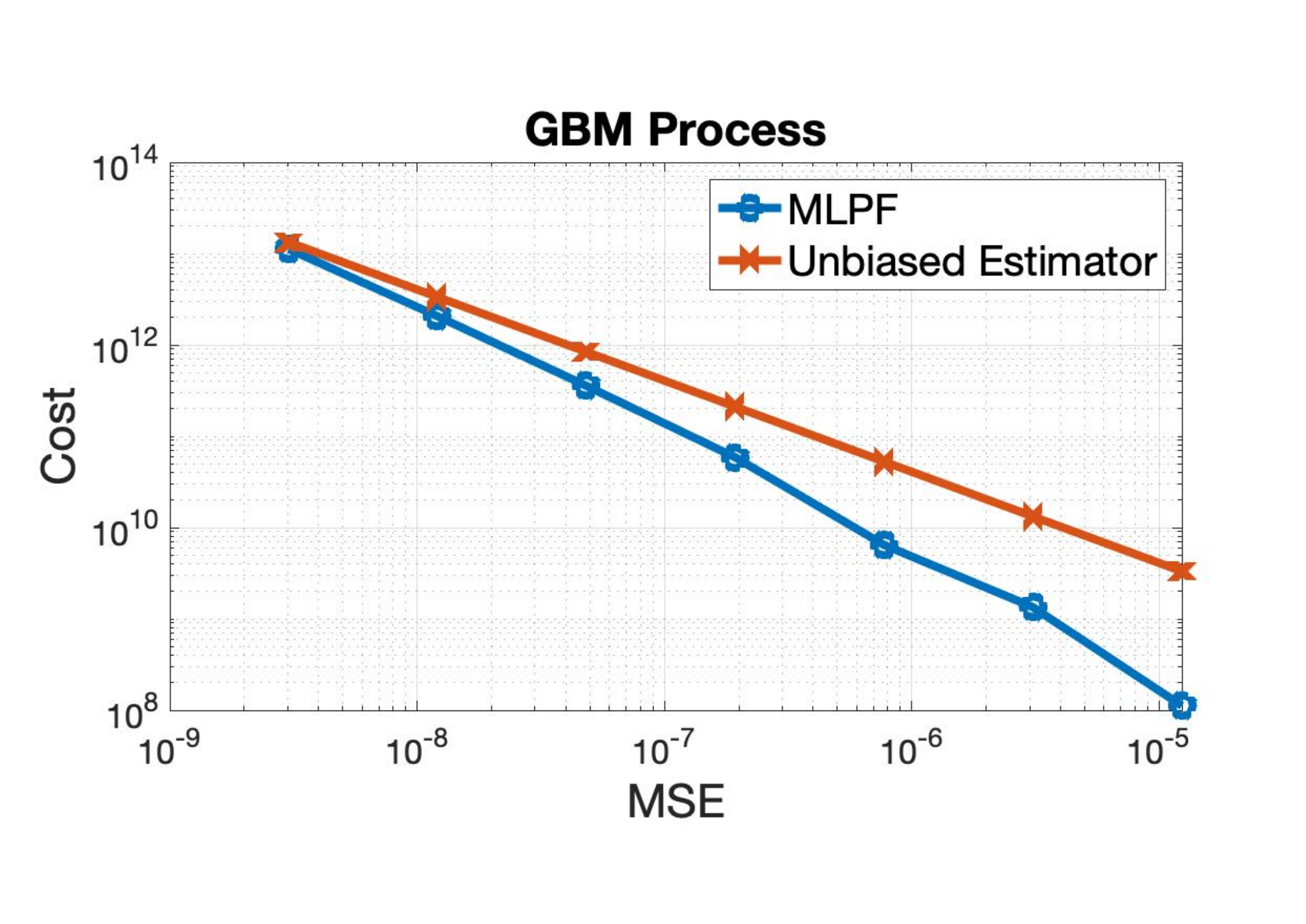}}
\caption{Cost against MSE plots. The UPF estimator is in red and MLPF estimator is in black.}\label{fig:res}	
\end{figure}

\begin{table}[h]
  \begin{center}
    \caption{Average Cost Ratio of Unbiased to MLPF}  \label{tab:tab1}
    \label{tab:table1}
    \begin{tabular}{l|r} 
      \textbf{Model} & \textbf{Cost Ratio}\\
      \hline
      OU & 3.80\\
      Langevin & 3.23\\
      NLD & 7.21\\
      GBM &  2.19\\
    \end{tabular}
  \end{center}
\end{table}

Our  results are presented in Figure \ref{fig:res} and Table \ref{tab:tab1} and concern the estimation of the filter at time 100.
Cost here is represented by the total number of discretized Euler updating steps in the simulation process. The average cost ratio (as in Table \ref{tab:tab1}) is the average  
of the ratio of the cost of \eqref{eq:ub_pf_est} 
to MLPF at same MSE levels over the last four (most precise discretization) values which are presented in Figure \ref{fig:res}. We also display rate plots (MSE vs Cost) for each model, where each plot has two lines in it representing the MLPF estimate calculated at $L\in\{1,2,\dots,7\}$ and 
the associated estimate \eqref{eq:ub_pf_est}. The plots are all on the $\log_{10}$-scale. 
 
 In Figure \ref{fig:res},  for the MLPF algorithm, we expect to see the non-asymptotic rate of $\frac{\log_{10}(\textrm{Cost})}{\log_{10}(\textrm{MSE})}$, which will be different from the asymptotic rate due to dominating effect of terms that are of order one; as the discretization level $l$ grows, the rate will converge to the asymptotic rate (for instance as in \cite{mlpf}). This effect is present in the simulation results, which shows a changing slope (this is more obvious for diffusion models with non-constant diffusion coefficient) for the line of the MLPF estimator.

We observe from the simulation results that 
the cost for  \eqref{eq:ub_pf_est} 
is higher than MLPF (the cost ratio range from 3 to 7) at the same MSE level. 
However parallel computing could make our estimator more appealing by reducing the actual computing time. If one has access to $K$ computers, then running  \eqref{eq:ub_pf_est} 
in parallel will reduce the computing time by a factor of $K$. For the Langevin or OU model, we need only a $K$ larger than 3 and the actual computing time using  \eqref{eq:ub_pf_est} 
to obtain estimates with the same MSE level will be shorter than that of MLPF. Even for the NLD model, a $K$ larger than 7 makes the unbiased estimator method more appealing since it requires less computing time to obtain the same MSE level. 
The parallelizability of the method allows us to embrace the possibilities provided by a new generation of massively parallel accelerator devices such as a graphics processing unit, Intel's Xeon Phi or even Field Programmable Gate Array. In the context of parallel computing, where one may typically have access to hundreds or thousands of nodes or more, 
each with multiple cores, \eqref{eq:ub_pf_est} 
could massively shorten the computing time, allowing us to generate very high accuracy estimators 
with speed that would be challenging to obtain using the MLPF by itself.

\subsubsection*{Acknowledgements}
A.J. \& F.Y. were supported by KAUST baseline funding.
K.J.H.L. \& A.J. were supported by the U.S. Department of Energy, Office of Science, 
Office of Advanced Scientific Computing Research (ASCR), 
under field work proposal number ERKJ333.

\appendix

\section{Proofs of Main Results}\label{sec:proofs}

The appendix is split into three sections. Section \ref{app:disc_conv} contains the proof of Proposition \ref{prop:euler_pred_conv}, Section \ref{app:pf_prf}
the proof of Proposition \ref{prop:pf_prop} and Section \ref{app:cpf_prf} the proof of Theorem \ref{theo:cpf_res}. Section \ref{app:cpf_prf} consists of a collection of
definitions and technical results which build up to the proof of the main theorem at the end.

\subsection{Proposition \ref{prop:euler_pred_conv}}\label{app:disc_conv}

\begin{proof}[Proof of Proposition \ref{prop:euler_pred_conv}]
We give the proof for the predictor; the filter follows directly from this.
Our proof is by induction on $n$. In the case $n=0$ we have for any $l\geq 0$
\begin{equation}\label{eq:weak_error}
|[\eta_0^l-\eta_0](\varphi)| =|[M^l-M](\varphi))(x^*)| \leq C\Delta_l\|\varphi\|_{\infty}
\end{equation}
for some $C<+\infty$ that does not depend upon $l$, where we have used the weak error for Euler approximations (see e.g.~\cite[eq.~(2.4)]{delm:01}), hence the initialization is verified.

Assume the result at rank $n-1$, then one has 
$$
[\eta_n^l-\eta_n](\varphi) = T_1 + T_2+T_3
$$
where
\begin{eqnarray*}
T_1 & := & \frac{1}{\eta_{n-1}^l(G_{n-1})}\eta_{n-1}^l(G_{n-1}[M^l-M](\varphi)) \\
T_2 & := & \frac{1}{\eta_{n-1}^l(G_{n-1})}[\eta_{n-1}^l-\eta_{n-1}](G_{n-1}M(\varphi)) \\
T_3 & := & \frac{\eta_{n-1}(G_{n-1}M(\varphi))}{\eta_{n-1}^l(G_{n-1})\eta_{n-1}(G_{n-1})}[\eta_{n-1}-\eta_{n-1}^l](G_{n-1}).
\end{eqnarray*}
For $T_2$ and $T_3$, as $l\rightarrow\infty$, they converge to zero by the induction hypothesis (recall $G_n\in\mathcal{B}_b(\mathsf{X})$ for every $n\geq 0$ by assumption). For $T_1$
$$
|\eta_{n-1}^l(G_{n-1}[M^l-M](\varphi))| \leq \|G_{n-1}\|_{\infty}\int_{\mathsf{X}}|[M^l-M](\varphi)(x)|\eta_{n-1}^l(dx).
$$
Applying \eqref{eq:weak_error} allows us to conclude that $T_1$ converges to zero as $l\rightarrow\infty$ (the denominator converges by the induction hypothesis). This completes the proof. 
\end{proof}

\subsection{Proposition \ref{prop:pf_prop}}\label{app:pf_prf}

To prove Proposition \ref{prop:pf_prop}, we give the following preliminary result.

\begin{prop}\label{prop:pf_prop_1}
For any $n\geq 0$ there exists a $C<+\infty$ such that for any $p\geq 0$, $N_p\geq 1$, $\varphi\in\mathcal{B}_b(\mathsf{X})$:
$$
\mathbb{E}_c\Big[\Big(\eta_{n}^{N_p,0}(\varphi)-\eta_n^0(\varphi)\Big)^2\Big] \leq \frac{C\|\varphi\|_{\infty}^2}{N_p}\Big(1+\frac{p^2}{N_p}\Big).
$$
\end{prop}

\begin{proof}
We have
$$
\mathbb{E}_c\Big[\Big(\eta_{n}^{N_p,0}(\varphi)-\eta_n^0(\varphi)\Big)^2\Big] = 
\sum_{k=0}^p\Big(\frac{(N_k-N_{k-1})}{N_p}\Big)^2\mathbb{E}\Big[\Big(\eta_{n}^{N_k-N_{k-1},0}(\varphi)-\eta_n^0(\varphi)\Big)^2\Big] + 
$$
$$
\sum_{k\neq s}
\Big(\frac{(N_k-N_{k-1})}{N_p}\Big)\Big(\frac{(N_s-N_{s-1})}{N_p}\Big)
\mathbb{E}[\eta_{n}^{N_k-N_{k-1},0}(\varphi)-\eta_n^0(\varphi)] \mathbb{E}[\eta_{n}^{N_s-N_{s-1},0}(\varphi)-\eta_n^0(\varphi)]
$$
where $\mathbb{E}$ is an expectation w.r.t.~law associated to a particle filter. \cite[Proposition 9.5.6]{delm:13} yields that
$$
\mathbb{E}_c\Big[\Big(\eta_{n}^{N_p,0}(\varphi)-\eta_n^0(\varphi)\Big)^2\Big] \leq C\|\varphi\|_{\infty}^2\Big(\sum_{k=0}^p\frac{(N_k-N_{k-1})}{N_p^2} + \sum_{k\neq s}\frac{1}{N_p^2}\Big)
$$
from which the proof can easily be concluded.
\end{proof}

\begin{proof}[Proof of Proposition \ref{prop:pf_prop}]
We have that
$$
\frac{\eta_{n}^{N_p,0}(G_n\varphi)}{\eta_{n}^{N_p,0}(G_n)}-\bar{\eta}_n^0(\varphi)
 = \frac{\eta_{n}^{N_p,0}(G_n\varphi)}{\eta_{n}^{N_p,0}(G_n)\eta_n^0(G_n)}\Big(\eta_n^0(G_n)-\eta_n^{N_p,0}(G_n)\Big) + 
\frac{1}{\eta_n^0(G_n)}\Big(\eta_{n}^{N_p,0}(G_n\varphi)-\eta_{n}^{0}(G_n\varphi)\Big).
$$
The proof can now easily be completed using Proposition \ref{prop:pf_prop_1} along with $\varphi\in\mathcal{B}_b(\mathsf{X})$.
\end{proof}

\subsection{Theorem \ref{theo:cpf_res}}\label{app:cpf_prf}

To prove Theorem \ref{theo:cpf_res}, we require some notations. Denote the sequence of non-negative kernels $\{Q_n^s\}_{n\geq 1}$, $s\in\{l,l-1\}$, $Q_n^s(x,dy) = G_{n-1}(x) M^s(x,dy)$ and for $B\in\mathcal{X}$, $x_p\in\mathsf{X}$
$$
Q_{p,n}^s(B)(x_p) = \int_{\mathsf{X}^{n-p}} \mathbb{I}_B(x_n) \prod_{q=p}^{n-1} Q_{q+1}^s(x_q,dx_{q+1})
$$
$0\leq p<n$ and in the case $p=n$, $Q_{p,n}^s$ is the identity operator.
Now denote for $0\leq p<n$, $s\in\{l,l-1\}$, $B\in\mathcal{X}$, $x_p\in\mathsf{X}$
$$
D_{p,n}^s(B)(x_p) = \frac{Q_{p,n}^s(\mathbb{I}_B - \eta_n^s(B))(x_p)}{\eta_p^s(Q_{p,n}^s(1))}
$$
in the case $p=n$, $D_{p,n}^s(B)(x)=\mathbb{I}_B-\eta_n^s(B)$.  Let $n\geq 1$, $B\in\mathcal{X}$ and $\mu\in\mathscr{P}(\mathsf{X})$ and define the probability measure:
$$
\Phi_n^l(\mu)(B) = \frac{\mu(G_{n-1}M^l(B))}{\mu(G_{n-1})}.
$$
Throughout the section $C$ is a finite and positive constant that does not depend upon $l$ and whose value may change on each appearance.

For $(\mathsf{X},\mathcal{X})$ a measurable space $(\mu,\nu)\in\mathcal{P}(\mathsf{X})^2$, the total variation distance  is written $\|\mu-\nu\|_{\textrm{tv}}=\sup_{A\in\mathcal{X}}|\mu(A)-\nu(A)|$. We start with a technical result that will be used below.

\begin{lem}\label{lem:av4}
Assume (A\ref{hyp:1}-\ref{hyp:2}). Then for any $n\geq 1$, $0\leq p < n$
there exist a $C<+\infty$ such that for any $(l,x,y)\in\mathbb{N}\times\mathsf{X}\times\mathsf{X}$, $\varphi\in \textrm{\emph{Lip}}(\mathsf{X})\cap\mathcal{B}_b(\mathsf{X})$
$$
|D_{p,n}^l(\varphi)(x)-D_{p,n}^{l-1}(\varphi)(y)| \leq
C\|\varphi\|_{\infty}\Big(\|x-y\|\wedge 1+\|\eta_p^l-\eta_p^{l-1}\|_{\textrm{\emph{tv}}}+\|\eta_n^l-\eta_n^{l-1}\|_{\textrm{\emph{tv}}}+|\|M_n^{l,l-1}\||\Big)
$$
where $C$ does not depend on $\eta_p^l,\eta_p^{l-1},\eta_n^l,\eta_n^{l-1}$ and $|\|M_n^{l,l-1}\||=\sup_{\{\varphi\in\textrm{\emph{Lip}}(\mathsf{X})\cap\mathcal{B}_b(\mathsf{X}:\|\varphi\|_{\infty}\leq 1\}}\sup_{x\in\mathsf{X}}|M_n^l(\varphi)(x)-M_n^{l-1}(\varphi)(x)|$.
\end{lem}

\begin{proof}
The proof is given in \cite{mlpf_new} and is omitted.
\end{proof}

The following result is for a CPF of $N$ samples with finite dimensional law $\mathbb{\check{P}}^N$ (expectations w.r.t.~$\mathbb{\check{P}}^N$ are written $\mathbb{\check{E}}^N$).

\begin{prop}\label{prop:cpf_res1}
Assume (A\ref{hyp:1}-\ref{hyp:2}). Then for any $n\in\mathbb{Z}^+$ there exists a $C<+\infty$ such that for any $(l,N)\in\mathbb{N}\times\mathbb{N}$, $\varphi\in\mathcal{B}_b(\mathsf{X})\cap\textrm{\emph{Lip}}(\mathsf{X})$:
$$
\Big|\mathbb{\check{E}}^N\Big[
[\eta_n^{N,l} - \eta_n^{N,l-1}](\varphi)
-[\eta_n^l-\eta_n^{l-1}](\varphi)\Big]\Big| \leq \frac{C\Delta_l^{\beta}\|\varphi\|_{\infty}}{N}
$$
where $\beta=\frac{1}{4}$ if $b$ is non-constant and $\beta=\frac{1}{2}$ if $b$ is constant.
\end{prop}

\begin{proof}
We have the following martingale plus remainder decomposition (see e.g.~\cite{mlpf_new,beskos,cpf_clt}):
\begin{eqnarray*}
[\eta_n^{N,l}(\varphi) - \eta_n^{N,l-1}](\varphi)-[\eta_n^l-\eta_n^{l-1}](\varphi) & = & \sum_{p=0}^n \frac{1}{\sqrt{N}}\{V_{p,n}^{N,l}(D_{p,n}^l(\varphi))-V_{p,n}^{N,l-1}(D_{p,n}^{l-1}(\varphi))\} \\ & & + \sum_{p=0}^{n-1} \{R_{p+1}^{N,l}(D_{p,n}^l(\varphi)) - R_{p+1}^{N,l-1}(D_{p,n}^{l-1}(\varphi))\}\\
V_{p,n}^{N,s}(\varphi) & = & \sqrt{N}[\eta_p^{N,s}-\Phi_p^s(\eta_{p-1}^{N,s})](\varphi) \\
R_{p+1}^{N,s}(D_{p,n}^s(\varphi))  & = & \frac{\eta_p^{N,s}(D_{p,n}^s(\varphi))}{\eta_p^{N,s}(G_p)}[\eta_p^{s}(G_p)-\eta_p^{N,s}(G_p)]
\end{eqnarray*}
with $s\in\{l,l-1\}$. Thus, we have
\begin{equation}\label{eq:mlpf_bias1}
\mathbb{\check{E}}^N[[\eta_n^{N,l}(\varphi) - \eta_n^{N,l-1}](\varphi)-[\eta_n^l-\eta_n^{l-1}](\varphi)] = \sum_{p=0}^{n-1}\mathbb{\check{E}}^N[R_{p+1}^{N,l}(D_{p,n}^l(\varphi)) - R_{p+1}^{N,l-1}(D_{p,n}^{l-1}(\varphi))].
\end{equation}
Then we have the decomposition
\begin{equation}\label{eq:mlpf_bias2}
R_{p+1}^{N,l}(D_{p,n}^l(\varphi)) - R_{p+1}^{N,l-1}(D_{p,n}^{l-1}(\varphi)) = \sum_{j=1}^4 T_j
\end{equation}
where
\begin{eqnarray*}
T_1 & = & -\frac{\eta_p^{N,l}(D_{p,n}^l(\varphi))}{\eta_p^{N,l}(G_p)}\Big(
[\eta_p^{N,l}-\eta_p^{N,l-1}](G_p) - [\eta_p^{l}-\eta_p^{l-1}](G_p)
\Big) \\
T_2 & = & \frac{[\eta_p^{l-1}-\eta_p^{N,l-1}](G_p)}{\eta_p^{N,l}(G_p)}\Big(
[\eta_p^{N,l}-\eta_p^{N,l-1}](D_{p,n}^l(\varphi)) - [\eta_p^{l}-\eta_p^{l-1}](D_{p,n}^l(\varphi)) 
\Big)\\
T_3 & = & \frac{[\eta_p^{l-1}-\eta_p^{N,l-1}](G_p)}{\eta_p^{N,l}(G_p)}[\eta_p^{N,l-1}-\eta_p^{N,l-1}](D_{p,n}^l(\varphi)-D_{p,n}^{l-1}(\varphi)) \\
T_4 & = & [\eta_p^{l-1}-\eta_p^{N,l-1}](G_p)\eta_p^{N,l-1}(D_{p,n}^{l-1}(\varphi))\Big(\frac{1}{\eta_p^{N,l}(G_p)}-\frac{1}{\eta_p^{N,l-1}(G_p)}\Big).
\end{eqnarray*}
We will bound the expectation for each of these terms in turn and then sum the bounds to conclude.

For $T_1$, via (A\ref{hyp:1}) and Cauchy-Schwarz, we have
$$
\mathbb{\check{E}}^N[T_1] \leq C\mathbb{\check{E}}^N[\eta_p^{N,l}(D_{p,n}^l(\varphi))^2]^{1/2}\mathbb{\check{E}}^N\Big[\Big(
[\eta_p^{N,l}-\eta_p^{N,l-1}](G_p) - [\eta_p^{l}-\eta_p^{l-1}](G_p)
\Big)^2\Big]^{1/2}.
$$
For the left hand term on the R.H.S., one can apply \cite[Proposition C.6.]{mlpf} and for the right hand term on the R.H.S.~\cite[Theorem C.4., Corollary D.6.]{mlpf} to yield
$$
\mathbb{\check{E}}^N[T_1] \leq \frac{C\|D_{p,n}^l(\varphi)\|_{\infty}\Delta_l^{\beta}}{N}.
$$
It easily follows from (A\ref{hyp:1}) that $\|D_{p,n}^l(\varphi)\|_{\infty}\leq C\|\varphi\|_{\infty}$ and hence that
\begin{equation}\label{eq:mlpf_bias3}
\mathbb{\check{E}}^N[T_1] \leq \frac{C\|\varphi\|_{\infty}\Delta_l^{\beta}}{N}.
\end{equation}

For $T_2$, one can follow almost an identical argument to $T_1$. First, note that $D_{p,n}^l(\varphi)\in\mathcal{B}_b(\mathsf{X})$ and second one can also use 
Lemma \ref{lem:av4} to verify that one would have a similar result to \cite[Theorem C.4., Corollary D.6.]{mlpf} when considering $D_{p,n}^l(\varphi)$. This yields
\begin{equation}\label{eq:mlpf_bias4}
\mathbb{\check{E}}^N[T_2] \leq \frac{C\Delta_l^{\beta}\|\varphi\|_{\infty}}{N}.
\end{equation}

For $T_3$, via (A\ref{hyp:1}) and Cauchy-Schwarz, we have
$$
\mathbb{\check{E}}^N[T_3] \leq C\mathbb{\check{E}}^N[[\eta_p^{l-1}-\eta_p^{N,l-1}](G_p)^2]^{1/2}\mathbb{\check{E}}^N[[\eta_p^{N,l-1}-\eta_p^{N,l-1}](D_{p,n}^l(\varphi)-D_{p,n}^{l-1}(\varphi))^2]^{1/2}.
$$
Applying \cite[Proposition C.6.]{mlpf} gives
$$
\mathbb{\check{E}}^N[T_3] \leq \frac{C\|D_{p,n}^l(\varphi)-D_{p,n}^{l-1}(\varphi)\|_{\infty}}{N}.
$$
Then by Lemma \ref{lem:av4}, \cite[Lemma D.2.]{mlpf}, \cite[eq.~(2.4)]{delm:01}, $\|D_{p,n}^l(\varphi)-D_{p,n}^{l-1}(\varphi)\|_{\infty}\leq C\Delta_l^{2\beta}\|\varphi\|_{\infty}$ and thus
\begin{equation}\label{eq:mlpf_bias5}
\mathbb{\check{E}}^N[T_3] \leq \frac{C\Delta_l^{2\beta}\|\varphi\|_{\infty}}{N}.
\end{equation}

For $T_4$, it follows that 
\begin{eqnarray}
\mathbb{\check{E}}^N[T_4] & = & T _5 + T_6 \label{eq:mlpf_bias10}\\
T_5 & = &  \mathbb{\check{E}}^N\Big[[\eta_p^{l-1}-\eta_p^{N,l-1}](G_p)\eta_p^{N,l-1}(D_{p,n}^{l-1}(\varphi))\Big(\frac{1}{\eta_p^{N,l}(G_p)}-\frac{1}{\eta_p^{N,l-1}(G_p)}-\Big(\frac{1}{\eta_p^{l}(G_p)}-\frac{1}{\eta_p^{l-1}(G_p)}\Big)\Big)\Big] \nonumber\\ 
T_6 & = &  \mathbb{\check{E}}^N\Big[[\eta_p^{l-1}-\eta_p^{N,l-1}](G_p)\eta_p^{N,l-1}(D_{p,n}^{l-1}(\varphi))\Big(\frac{1}{\eta_p^{l}(G_p)}-\frac{1}{\eta_p^{l-1}(G_p)}\Big)\Big]. \nonumber
\end{eqnarray}
We now need to control $T_5$ and $T_6$. Now, for $T_5$ applying Cauchy-Schwarz twice gives
$$
T_5 \leq \mathbb{\check{E}}^N[[\eta_p^{l-1}-\eta_p^{N,l-1}](G_p)^4]^{1/4}\mathbb{\check{E}}^N[\eta_p^{N,l-1}(D_{p,n}^{l-1}(\varphi))^4]^{1/4}\mathbb{\check{E}}^N\Big[\Big(\frac{1}{\eta_p^{N,l}(G_p)}-\frac{1}{\eta_p^{N,l-1}(G_p)}-\Big(\frac{1}{\eta_p^{l}(G_p)}-\frac{1}{\eta_p^{l-1}(G_p)}\Big)\Big)^{2}\Big]^{1/2}
$$
Applying \cite[Proposition C.6.]{mlpf} twice, gives
\begin{equation}\label{eq:mlpf_bias6}
T_5 \leq \frac{C\|\varphi\|_{\infty}}{N}\mathbb{\check{E}}^N\Big[\Big(\frac{1}{\eta_p^{N,l}(G_p)}-\frac{1}{\eta_p^{N,l-1}(G_p)}-\Big(\frac{1}{\eta_p^{l}(G_p)}-\frac{1}{\eta_p^{l-1}(G_p)}\Big)\Big)^{2}\Big]^{1/2}.
\end{equation}
Now, by Minkowski
\begin{equation}\label{eq:mlpf_bias7}
\mathbb{\check{E}}^N\Big[\Big(\frac{1}{\eta_p^{N,l}(G_p)}-\frac{1}{\eta_p^{N,l-1}(G_p)}-\Big(\frac{1}{\eta_p^{l}(G_p)}-\frac{1}{\eta_p^{l-1}(G_p)}\Big)\Big)^{2}\Big]^{1/2}
\leq  T_7 + T_8
\end{equation}
where
\begin{eqnarray*}
T_7 & = & \mathbb{\check{E}}^N\Big[\Big(\frac{\eta_p^{N,l-1}(G_p)-\eta_p^{N,l}(G_p)-(\eta_p^{l-1}(G_p)-\eta_p^{l}(G_p))}{\eta_p^{N,l}(G_p)\eta_p^{N,l-1}(G_p)}\Big)^{2}\Big]^{1/2} \\
T_8 & = & |\eta_p^{l-1}(G_p)-\eta_p^{l}(G_p)|\mathbb{\check{E}}^N\Big[\Big(
\frac{\eta_p^{l-1}(G_p)(\eta_p^{l}(G_p)-\eta_p^{N,l}(G_p)) + \eta_p^{N,l}(G_p)(\eta_p^{l-1}(G_p)-\eta_p^{N,l-1}(G_p))}{\eta_p^{N,l}(G_p)\eta_p^{N,l-1}(G_p)\eta_p^{l}(G_p)\eta_p^{l-1}(G_p)}\Big)^2\Big]^{1/2}.
\end{eqnarray*}
For $T_7$, by (A\ref{hyp:1}), \cite[Theorem C.4., Corollary D.6.]{mlpf}
$$
T_7 \leq \frac{C\Delta_l^{\beta}}{N^{1/2}}.
$$
For $T_8$, by (A\ref{hyp:1}), Minkowski, \cite[Proposition C.6.]{mlpf} (twice) and \cite[Lemma D.2.]{mlpf} (for $|\eta_p^{l-1}(G_p)-\eta_p^{l}(G_p)|$)
$$
T_8 \leq \frac{C\Delta_l^{2\beta}}{N^{1/2}}.
$$
Noting \eqref{eq:mlpf_bias6} and \eqref{eq:mlpf_bias7} we then have
\begin{equation}\label{eq:mlpf_bias8}
T_5 \leq \frac{C\Delta_l^{\beta}\|\varphi\|_{\infty}}{N^{3/2}}.
\end{equation}
For $T_6$, via Cauchy-Schwarz
$$
T_6 \leq \frac{|\eta_p^{l-1}(G_p)-\eta_p^{l}(G_p)|}{\eta_p^{l-1}(G_p)\eta_p^{l}(G_p)}\mathbb{\check{E}}^N[\eta_p^{l-1}-\eta_p^{N,l-1}](G_p)^2]^{1/2}\mathbb{\check{E}}^N[\eta_p^{N,l-1}(D_{p,n}^{l-1}(\varphi))^2]^{1/2}
$$
Applying \cite[Proposition C.6.]{mlpf} twice, \cite[Lemma D.2.]{mlpf} and (A\ref{hyp:1})
\begin{equation}\label{eq:mlpf_bias9}
T_6 \leq \frac{C\Delta_l^{\beta}\|\varphi\|_{\infty}}{N}.
\end{equation}
Noting \eqref{eq:mlpf_bias10}, along with \eqref{eq:mlpf_bias8} and  \eqref{eq:mlpf_bias9} gives
\begin{equation}\label{eq:mlpf_bias11}
\mathbb{\check{E}}^N[T_4] \leq \frac{C\Delta_l^{\beta}\|\varphi\|_{\infty}}{N}.
\end{equation}

The proof is concluded by noting \eqref{eq:mlpf_bias1},  \eqref{eq:mlpf_bias2} along with \eqref{eq:mlpf_bias3}, \eqref{eq:mlpf_bias4}, \eqref{eq:mlpf_bias5} and \eqref{eq:mlpf_bias11}.
\end{proof}

\begin{prop}\label{prop:cpf_prop_1}
Assume (A\ref{hyp:1}-\ref{hyp:2}). Then for any $n\in\mathbb{Z}^+$ there exists a $C<+\infty$ such that for any $(l,p)\in\mathbb{N}\times\mathbb{Z}^+$, $N_p>N_{p-1}>\cdots>N_0\geq 1$, $\varphi\in\mathcal{B}_b(\mathsf{X})\cap\textrm{\emph{Lip}}(\mathsf{X})$:,
$$
\mathbb{\check{E}}_c\Big[\Big([\eta_n^{N_{0:p},l}-\eta_n^{N_{0:p},l-1}](\varphi)-[\eta_n^{l}-\eta_n^{l-1}](\varphi)\Big)^2\Big] \leq \frac{C\Delta_l^{\beta}\|\varphi\|_{\infty}^2}{N_p}\Big(1+\frac{p^2}{N_p}\Big)
$$
where $\beta=\frac{1}{2}$ if $b$ is non-constant and $\beta=1$ if $b$ is constant.
\end{prop}

\begin{proof}
The proof is essentially the same as Proposition \ref{prop:pf_prop_1}, except that one uses \cite[Theorem C.4., Corollary D.6.]{mlpf} and Proposition \ref{prop:cpf_res1} in place of  \cite[Proposition 9.5.6]{delm:13}.
\end{proof}

\begin{proof}[Proof of Theorem \ref{theo:cpf_res}]
The proof follows by \cite[Lemma C.5.]{mlpf}, Proposition \ref{prop:cpf_prop_1}, (A\ref{hyp:1}), \cite[Proposition C.6.]{mlpf} and \cite[Lemma D.2.]{mlpf}. The details are omitted.
\end{proof}


\begin{thebibliography}{99}

\bibitem{mlpf_new}
{\sc Ballesio}, M., {\sc Jasra}, A., {\sc Von Schwerin}, E., \& {\sc Tempone}, R.~(2020).
A Wasserstein coupled particle filter for multilevel estimation. Technical Report.


\bibitem{beskos}
{\sc Beskos}, A., 
{\sc Jasra}, A., 
{\sc Law}, K. J. H., 
{\sc Tempone}, R., \& 
{\sc Zhou}, Y.~ (2017). 
Multilevel Sequential Monte Carlo samplers. {\em Stoch. Proc. Appl.}, {\bf 127}, 1417-1440.

\bibitem{glynn}
{\sc Blanchet}, J.,  {\sc Glynn}, P. \& {\sc Pei}, Y.~(2019). Unbiased Multilevel Monte Carlo. arXiv preprint.

\bibitem{cappe}
{\sc Cappe}, O., {\sc Moulines}, E. \& {\sc Ryden}, T.~(2005).
\emph{Inference in Hidden Markov models}. Springer: New York.


\bibitem{delm:13}
{\sc Del Moral}, P.~(2013). \textit{Mean Field Simulation for Monte Carlo Integration}. Chapman \& Hall: London.


\bibitem{delm:01}
{\sc Del Moral}, P., {\sc Jacod}, J. \& {\sc Protter}, P.~(2001).
The Monte Carlo method for filtering with discrete-time observations.
\emph{Probab. Theory Rel. Fields}, {\bf 120}, 346--368.


\bibitem{fearn}
{\sc Fearnhead}, P., 
{\sc Papaspiliopoulos}, O. \& {\sc Roberts}, G. O.~(2008). 
Particle filters for partially observed diffusions. \emph{J. R. Stat. Soc. Ser. B} {\bf 70}, 755--777.

\bibitem{giles}
{\sc Giles}, M.~B.~(2008).
Multilevel Monte Carlo path simulation.
\emph{Op. Res.}, {\bf 56}, 607-617.

\bibitem{giles2}
{\sc Giles}, M.~B.~(2015). 
Multilevel monte carlo methods. 
\emph{Acta Numerica}, 
{\bf 24}, 259-328.

\bibitem{glynn2}
{\sc Glynn}, P.~W., \& 
{\sc Rhee}, C.~H.~(2014). 
Exact estimation for Markov chain equilibrium expectations. 
\emph{J. Appl. Probab.}, 
{\bf 51}, 
377-389.

\bibitem{heinrich}
{\sc Heinrich}, S.~(2001). Multilevel Monte Carlo methods. In \emph{Large Scale Scientific Computing}, 
Springer: New York.


\bibitem{jacob2}
{\sc Jacob}, P., {\sc Lindsten}, F. \& {\sc Sch\"on}, T.~(2020). Smoothing with couplings of conditional particle filters.
\emph{J. Amer. Statist. Assoc.} (to appear).

\bibitem{cpf_clt}
{\sc Jasra}, A., \& {\sc Yu}, F.~(2018). Central limit theorems for coupled particle filters. arXiv:1810.04900.


\bibitem{mlpf}
{\sc Jasra}, A., {\sc Kamatani}, K., {\sc Law} K. J. H. \& {\sc Zhou}, Y.~(2017). 
Multilevel particle filters. \emph{SIAM J. Numer. Anal.}, {\bf 55}, 3068-3096.

\bibitem{mcl}
{\sc McLeish}, D.~(2011). A general method for debiasing a Monte Carlo estimator. \emph{Monte Carlo Meth. Appl.}, {\bf 17}, 301--315.

\bibitem{rhee}
{\sc Rhee}, C. H. \& {\sc Glynn}, P.~(2015). Unbiased estimation with square root convergence for SDE models. \emph{Op. Res.}~{\bf 63}, 1026--1043. 

\bibitem{vihola}
{\sc Vihola}, M.~(2018). Unbiased estimators and multilevel Monte Carlo. \emph{Op. Res.}, {\bf 66}, 448--462.

\end{thebibliography}
\end{document}